\documentclass{amsart}

\usepackage{amsmath,amsfonts,amssymb}
\usepackage[colorlinks=true]{hyperref}
\hypersetup{urlcolor=blue, citecolor=blue}
\usepackage{geometry}
\usepackage{enumerate}
\usepackage{graphicx}
\usepackage{subcaption}
 \usepackage{dsfont}
 \usepackage{bm}

\newtheorem{thm}{Theorem}[section]
\newtheorem{lem}[thm]{Lemma}
\newtheorem{prop}[thm]{Proposition}
\newtheorem{coro}[thm]{Corollary}
\newtheorem{Def}[thm]{Definition}

\newtheorem{rem}[thm]{Remark}

\def\R{\mathbb{R}}

\def\be{\begin{equation}}
\def\ee{\end{equation}}
\def\g{{\bf g}}
\def\n{{\boldsymbol n}}
\def\p{\partial}
\def\grad{\boldsymbol{\nabla}}
\def\div{\grad\cdot}

\def\O{\Omega}
\def\G{\Gamma}

\def\sig{\sigma}

\def\a{\alpha}
\def\eps{\epsilon}
\def\x{{\boldsymbol x}}

\def\br{{\boldsymbol r}}
\def\d{{\rm d}}
\def\ov#1{\overline{#1}}

\def\wt#1{\widetilde{#1}}

\def\bphi{{\boldsymbol{\phi}}}

\def\beps{{\boldsymbol{\varepsilon}}}

\def\bF{{\boldsymbol  F}}

\def\bI{{\boldsymbol  I}}

\def\bR{{\boldsymbol  R}}

\def\bU{{\boldsymbol  U}}

\def\bX{{\boldsymbol  X}}
\def\bY{{\boldsymbol  Y}}

\def\bu{{\boldsymbol  u}}

\def\bv{{\boldsymbol v}}
\def\bx{{\boldsymbol x}}

\def\bw{{\boldsymbol w}}
\def\0{{\bf 0}}
\def\1{{\bf 1}}

\def\bbI{{\mathbb I}}

\def\bbK{{\mathbb K}}

\def\Aa{\mathcal{A}}

\def\Bb{\mathcal{B}}

\def\Dd{\mathcal{D}}

\def\Ff{\mathcal{F}}
\def\Gg{\mathcal{G}}

\def\Jj{\mathcal{J}}
\def\Kk{\mathcal{K}}

\def\Oo{\mathcal{O}}

\def\Rr{\mathcal{R}}
\def\Ss{\mathcal{S}}
\def\Tt{\mathcal{T}}
\def\Uu{\mathcal{U}}
\def\Vv{\mathcal{V}}
\def\Ww{\mathcal{W}}

\def\bUu{\boldsymbol{\mathcal{U}}}

\def\nn{\nonumber}

\def\bp{{\boldsymbol{p}}}

\def\bPi{{\boldsymbol{\Pi}}}

\def\bsig{{\boldsymbol{\sig}}}
\def\bPhi{{\boldsymbol{\Phi}}}

\def\ind{{\mathds 1}}
\def\Tr{\operatorname{Tr}}
\def\bHh{{\boldsymbol{\mathcal H}}}

\def\gravmech{{\boldsymbol f}}

\def\triplenorm#1{|\mskip -2 mu|\mskip -2 mu|#1|\mskip -2mu |\mskip -2mu|}

\newcounter{cst}
\newcommand{\ctel}[1]{C_{\refstepcounter{cst}\label{#1}\thecst}}
\newcommand{\cter}[1]{C_{\ref{#1}}}

\def\clemrev#1{{#1}}

\begin{document}

\title{A global existence result for weakly coupled two-phase poromechanics}
\author{Jakub W. Both}
\address{Jakub W. Both (\href{mailto:Jakub.Both@uib.no}{\tt Jakub.Both@uib.no}): Center for Modeling of Coupled Subsurface Dynamics, Department of Mathematics, University
of Bergen, P.O. Box 7803, 5020 Bergen, Norway}
\author{Cl\'ement Canc\`es}       
\address{Clément Cancès (\href{mailto:clement.cances@inria.fr}{\tt clement.cances@inria.fr}): Inria, Univ. Lille, CNRS, UMR 8524 - Laboratoire Paul Painlevé, F-59000 Lille, France.}


\maketitle

\begin{abstract}
 Multiphase poromechanics describes the evolution of multiphase flow in deformable porous media. Mathematical models for such multiphysics system are inheritely nonlinear, potentially degenerate and fully coupled systems of partial differential equations. In this work, we present a thermodynamically consistent multiphase poromechanics model falling into the category of Biot equations and obeying to a generalized gradient flow structure. It involves capillarity effects, degenerate relative permeabilities, and gravity effects. In addition to established models it introduces a Lagrange multiplier associated to a bound constraint on the effective porosity in particular ensuring its positivity. We establish existence of global weak solutions under the assumption of a weak coupling strength, implicitly utilizing the gradient flow structure, as well as regularization, a Faedo-Galerkin approach and compactness arguments. This comprises the first global existence result for multiphase poromechanics accounting for degeneracies that are consistent with the multiphase nature of the flow.
\end{abstract}


\section{Two-phase poromechanics: motivation, formulation and main result}\label{sec:2}

\subsection{Introduction}
Poromechanics, modeled by the prototypical Biot equations, describes the two-way coupled interaction between flow in porous media and its macro- and microscopic deformation. Since the seminal works in mathematical modelling by von Terzaghi~\cite{terzaghi1925erdbaumechanik} and Biot~\cite{Biot41}, the study of poromechanics has been extended to various applications across engineering~\cite{Coussy04}. With relevance ranging from geotechnical, environmental, over industrial to biomedical engineering, poromechanics plays a paramount role.

In many engineering systems, as e.g. subsurface reservoirs for geological CO$_2$ storage or geothermal energy, the simultaneous presence of multiple fluids introduces a range of physical phenomena which requires nonlinear and degenerate mathematical models. To capture particular multiphase effects as varying saturations, capillary pressure, drying shrinkage, pore pressure changes and finally related deformation and subsidence, the classical linearized Biot equations are not sufficient and require the introduction of a range of nonlinear constitutive relations. For this, mathematical models for multiphase poromechanics have been introduced combining separate components of multiphase flow and poroelasticity modeling~\cite{Coussy04}. The general model structure has been also studied from a thermodynamical point of view~\cite{seguin2019multi}. Furthermore, there exist many works on the numerical approximation of such systems, e.g.,~\cite{BBDM21,both2019anderson} and they are employed by practitioners~\cite{jha2014coupled}.

Despite an increased interest in the modelling and numerics communities, multiphase
poromechanics has been studied in much less detail and with less rigor than the much simpler linearized
Biot equations for single fluids and several extensions -- no well-posedness results exist for two-phase poromechanics models. In contrast, in the seminal work~\cite{auriault1977sanchez}, the first well-posedness results for the linear Biot equations have been established, deriving existence of strong solutions. A series of works has followed establishing weak solution~\cite{vzenivsek1984existence}, well-posedness from the perspectives of semi-group theory~\cite{showalter2000diffusion} as well as generalized gradient flow theory~\cite{both2019gradient}. The well-posedness for mixed-dimensional extensions to fractured media under contact have been established in~\cite{boon2023mixed} utilizing monotone operator theory. Additional physics have been considered and rigorously investigated, as linear poroviscoelastic effects~\cite{bociu2023mathematical} as well as the coupling with Stokes' equations in a classical fluid-structure-interaction fashion~\cite{ambartsumyan2019nonlinear,bociu2021multilayered}. Analytical results for nonlinear extensions include single-phase flow accounting for compressibility and positivity-preserving porosity~\cite{van2023mathematical}, thermal effects~\cite{van2019thermoporoelasticity,brun2019well}, as well as displacement-depending permeability laws~\cite{bociu2016analysis}. Finally, closely related to this study, it is worth stressing, that the existence of weak solutions for unsaturated poromechanics models has been established in~\cite{showalter2001partially,both2021global}, where unsaturated media unlike true multiphase systems are partially saturated by one active fluid governing the displacement of a second passive fluid. 

The model we consider here accounts for the motion of two incompressible phases in a porous medium, modeled by the Darcy-Muskat law and some capillary pressure law, see e.g. \cite{BB90}. 
In particular, the system degenerates as one fluid phase vanishes, leading to mathematical difficulties making suitable mathematical reformulations involving the Kirchhoff transform and Chavent global pressure necessary (see for instance~\cite{CJ86, Chen01}). 
The porous matrix is supposed to be elastic (small deformation is postulated) and interacts with the fluid by the mean of the equivalent pressure postulated by Coussy~\cite{Coussy04}. 
As we restrict to linear relative permeabilities, the Coussy equivalent pressure coincides with the Chavent global pressure. Another important difficulty comes from the fact that linear elasticity does not prevent the porosity to become nonpositive. To maintain the model in its regime of validity, we incorporated some Lagrange multiplier $\chi$, so that we do not have to assume the porosity to be positive as for instance in~\cite{BBDM21}. Furthermore, we allow for the permeability to depend on the porosity in the line of Kozeny~\cite{Kozeny27} and Carman~\cite{Carman37}. 
Finally, state-dependent gravitational forces are incorporated acting both on the bulk as well as on the fluid phases. To the best of the authors' knowledge, the presented result is the first-ever existence result for multiphase poromechanics in presence of physically relevant degenerate mobilities. Our results comes however with restrictions. Constant Biot coefficients and Biot modulus are assumed for simplicity, as well as specific (but physically meaningful) boundary conditions on the displacement. We also require some weak-coupling condition. Assumptions are detailed and discussed in Section~\ref{ssec:2.weak} later on. The extension of existence of weak solutions for tightly coupled and/or heterogeneous in space systems with discontinuous characteristics remains an open problem.

\subsection{The governing equations}\label{ssec:2.eq}
Let us first start by stating the equations governing the motions of two immiscible fluids --- a wetting one labeled with subscript $w$ and a 
non-wetting one labeled by $n$ --- in a porous medium represented by some bounded open set $\Omega \subset \R^d$. 
We remain sloppy here concerning regularity, which will be made precise later on. 
For $\a\in \{n,w\}$, denote by $\rho_\a$ and $\mu_\a$ the (constant) density and viscosity of the phase $\a$, and by $\g$ the gravity vector (pointing downwards). In addition, as in the single-phase Biot model~\cite{Biot41}, the porous matrix is assumed to be deformable; let $\bu$ denote the macroscopic displacement of the matrix wrt.\ $\Omega$.
\begin{subequations}\label{eq:1}
The classical multiphase Darcy law, with the hydrostatic phase pressure extended to the deformable case in a thermodynamically consistent way, then writes
\be\label{eq:1.cons}
\p_t \phi_\a  - \div\left( \frac{s_\a}{\mu_\a} \bbK(\phi) \grad \left( p_\a  - \rho_\a \g \cdot (\bx + \bu)  \right) \right) = 0
\ee
where, for volume densities $\bphi = (\phi_n, \phi_w)$ of the two phases, we denote by 
$\phi = \phi_n + \phi_w$ the porosity, and by $s_\a = \frac{\phi_\a}{\phi}$ the saturation (or volume fraction) of the phase 
$\alpha$. The intrinsic permeability $\bbK$ of the porous medium may depend on the porosity $\phi$, 
whereas we restrict our purpose to linear relative permeabilities for technical reasons that will appear
 in what follows. The pressure of the phase $\alpha \in \{n,w\}$ decomposes into three contributions:
\be\label{eq:1.p}
p_\a = \hat p_\a + \pi +\chi. 
\ee
In the above  right-hand  side, only the first term $\hat p_\alpha$ depends on $\alpha$. It is related to $\bphi$, and more 
specifically to the saturations, via the formulas
\be\label{eq:1.hp}
\hat p_n  = \gamma(s_n) + s_w \gamma'(s_n) \quad\text{and}\quad \hat p_w = \gamma(s_n) - s_n \gamma'(s_n)
\ee
where $\gamma: [0,1] \to \R_+$ is convex and increasing, and shall be interpreted as the antiderivative of the capillary pressure function.
The last term $\chi$, the introduction of which being somehow artificial, shall be thought as a Lagrange multiplier ensuring that the porosity will not leave the range of validity of the model $\phi \in [\phi_\flat, \phi^\sharp]$ for some constants $0 < \phi_\flat < \phi^\sharp< 1$. Therefore, $\chi$ is related to $\phi$ through the maximal monotone graph $\ov \chi = \partial \ind_{[\phi_\flat,\phi^\sharp]}$ by 
\be\label{eq:1.chi}
\chi \in \ov \chi(\phi) \quad \text{with}\; \ov \chi(\phi) = \begin{cases} 
0 & \text{if} \; \phi_\flat < \phi < \phi^\sharp, \\
(-\infty,0] & \text{if}\; \phi = \phi_\flat, \\
[0,+\infty) & \text{if}\; \phi = \phi^\sharp.
\end{cases}
\ee
\end{subequations}
Note in particular that since $s_n + s_w = 1$, one recovers the capillary pressure relation 
\[
p_n - p_w = \gamma'(s_n). 
\]
Moreover, we deduce from~\eqref{eq:1.hp} that $s_n  \hat p_n + s_w \hat p_w = \gamma(s_n)$, so that \eqref{eq:1.p} yields
\[
s_n p_n + s_w  p_w  - \gamma(s_n) = \pi  +  \chi. 
\]
In particular, on the set $\{\phi_\flat < \phi < \phi^\sharp \}$, where the Lagrange multiplier $\chi$ vanishes, $\pi$ coincides 
with the equivalent pressure introduced by Coussy~\cite{Coussy04}, and is therefore referred to as the Coussy pressure. 
It encodes the pressure felt by the porous matrix surrounding the fluid. 

We assume small displacements $\bm{u}$, describing the macroscopic deformation of the bulk, so that we can stick to linear elasticity. Moreover, we assume the system to be quasi-static, i.e. the mechanical response of the matrix is supposed to be instantaneous, so that mechanical forces remain at equilibrium.
Denoting by $\ov \bsig$ the Cauchy stress tensor, decomposing into effective stress $\bsig$ and pore pressure contribution $\pi$~\cite{Biot41,Coussy04}
\be
\ov \bsig = \bsig - b \pi \bbI,
\ee
and by $\gravmech$ a body force, the balance of linear momentum $-\div \ov \bsig = \gravmech$ writes
\begin{subequations}\label{eq:2}
\be\label{eq:2.forces}
\div \bsig =   b \, \grad \pi- \gravmech, 
\ee
where $b\in (0,1]$ is the so-called Biot coefficient.
The effective stress tensor $\bsig$ and the strain tensor $\beps(\bu) = (\grad \bu + \grad \bu^T)/2$ are related through the isotropic Hooke law
\be\label{eq:2.hooke}
 \bsig = 2 \mu\,  \beps(\bu) + \lambda \div \bu\, \bbI,
 \ee 
 \end{subequations}
where $\lambda, \mu>0$ are the Lamé coefficients. The body force $\gravmech$ may depend on the fluid distribution, e.g., the gravitational force typically incorporates an effective volume-averaged bulk density
\begin{subequations}\label{eq:2.grav}
\be\label{eq:2.gravitational_force}
\gravmech_\g(\bphi) = \left(\phi \sum_{\a \in\{n,w\}} s_\a \rho_\alpha + (1-\phi) \rho_s(\phi)\right) \g = \left( \sum_{\a \in\{n,w\}} \phi_\a \rho_\a + (1-\phi_r) \rho_{s,r}\right)\g,
\ee
where $\rho_s$ denotes the rock density set to be a function of the porosity, defined through
\be\label{eq:2.rock}
  \left(1-\phi\right) \rho_s(\phi)  = \left(1-\phi_r\right)\rho_{s,r}, 
\ee
and thus satisfying satisfying mass conservation $\frac{d}{dt} \int_E (1-\phi) \rho_s(\phi) = 0$ for any measurable subset $E$ of $\O$, cf.~\cite{Coussy04}. From now on, we consider for $\gravmech$
\be
\gravmech(\bphi) = \gravmech_\g(\bphi) + \gravmech_\mathrm{ext}
\ee
for some additional external, state-independent body force $\gravmech_\mathrm{ext}$.
\end{subequations} 

Beside the macroscopic deformation of the matrix encoded by its displacement, the porous structure is assumed to be compressible. 
This is encoded by the parameter $\theta$ representing the microscopic compression of the solid grains of the porous structure. 
It simply relates to the Coussy pressure by
\be\label{eq:3.theta}
M\theta = \pi 
\ee
with $M>0$ being referred to in the literature as the Biot modulus. 

The last equation set on the time-and-space bulk $\R_+\times\O$ is a constraint on the fact that the fluid and the solid have to share the available space, leading to
\be\label{eq:constraint}
\phi  - b \, \div \bu - \theta = \phi_r
\ee
for some spatially varying reference porosity with values $\phi_r(\bx) \in (\phi_\flat,\phi^\sharp)$, $\bx\in\Omega$,  representing the porosity at rest. 

To close the system, we prescribe some initial conditions 
\be\label{eq:init}
\bphi_{|_{t=0}} = ({\phi_n}_{|_{t=0}}, {\phi_w}_{|_{t=0}})= \bphi^0 = (\phi_n^0, \phi_w^0)
\ee
on the fluid contents with 
\[
\phi_\flat \leq \phi^0 = \phi_n^0 + \phi_w^0 \leq \phi^\sharp.
\]
Since the mechanical  equilibration is instantaneous, the initial displacement $\bu^0$ and microscopic deformations $\theta^0$ are derived from $\bphi^0$ as the minimizer of some mechanical energy under the constraint~\eqref{eq:constraint}, see (H\ref{H.init}) later on. 
Boundary conditions of mixed type are considered. More precisely, given a partition $\Gamma^N, \Gamma^D$ of $\p\O$, and denoting 
by $\n$  the normal to $\p\O$ outward w.r.t. $\O$, 
\begin{subequations}\label{eq:BC}
we prescribe 
\be\label{eq:BC.u}
\bu = \bm{0} \quad \text{on}\; \R_+\times \G^D \qquad\text{and}\qquad \bsig\cdot \n = b \, \pi \n \quad \text{on}\; \R_+\times \G^N
\ee
for the solid mechanics equations~\eqref{eq:2}, i.e., homogeneous boundary conditions for displacement and traction in terms of the effective stress, while the fluid part~\eqref{eq:1} is complemented by setting
\be\label{eq:BC.phi.N}
- \frac{s_\a}{\mu_\a} \bbK(\phi)\grad \left( p_\a  - \rho_\a \g \cdot(\bx + \bu) \right) \cdot \n = \0 \quad \text{on}\; \R_+\times \G^N,
\ee
and 
\be\label{eq:BC.phi.D}
p_n  = p_n^D  \quad \text{and}\quad p_w  = p_w^D \quad\text{on}\; \R_+\times \G^D.
\ee
\end{subequations}

\begin{rem}[Non-homogeneous Dirichlet data for the displacement]
 The case of non-homogeneous Dirichlet data $\bu^D$ for the displacement in~\eqref{eq:BC.u} can be reduced in a standard way to the case of homogeneous data, considered above. Indeed, in the momentum balance equation~\eqref{eq:2.forces}, using a lifting, its contribution can be incorporated in the constant, external body force $\gravmech_\mathrm{ext}$, whereas, in the porosity constraint~\eqref{eq:constraint}, the contribution is included in the definition of the reference porosity $\phi_r$, consistently with its character representing the medium at rest. 
\end{rem}

\subsection{A thermodynamic viewpoint}\label{ssec:2.thermo}
The equations presented in Section~\ref{ssec:2.eq} have strong connections with thermodynamics. 
In particular, the dynamics prescribed by~\eqref{eq:1}--\eqref{eq:BC} can be interpreted as some generalized gradient flow of the Helmholtz free energy augmented with some potential energy related to body forces and in particular to gravity.

More precisely, to a set $\bX = (\bphi, \beps, \theta)$ of primary unknowns (here $\beps = \beps(\bu)$), we associate the Helmholtz free energy defined as the sum of three contributions 
\be\label{eq:Helmholtz}
\Ff(\bX) = \Ff_f(\bphi) + \Ff_s(\beps, \theta) + \Ff_c(\bX). 
\ee
The energy $\Ff_f(\bphi)$ associated to the fluid is defined as 
\be\label{eq:Fff}
\Ff_f(\bphi) = \int_\O \left\{ \phi \gamma(s_n) + \ind_{[\phi_\flat,\phi^\sharp]}(\phi) \right\}, 
\ee
where $\ind_C(v) = 0$ if $v \in C$ and $+\infty$ otherwise. 

The mechanical energy associated to the porous matrix deformation is given by 
\be\label{eq:Ffs}
\Ff_s(\beps,\theta) = \int_\O \left\{\frac{M}2 \theta^2 + \mu \, \beps:\beps + \frac{\lambda}2 |\Tr\beps|^2 \right\}
\ee
with $\Tr\beps = \div \bu$.

The last term $\Ff_c(\bX)$ enforces the constraint~\eqref{eq:constraint} to hold (almost) everywhere in $\O$. 
We introduce the (linear thus) convex set $\Kk=\{\bX\, | \, \phi - b \Tr\beps - \theta = 1-\phi_r^\star\}$, then we set 
\be\label{eq:Ffc}
\Ff_c(\bX) = \int_\O \ind_\Kk(\bX)  = \sup_{w} \int_\O \left(\phi - \phi_r - b\,  \Tr\beps- \theta \right) w.
\ee

In addition, the potential energy of the bulk associated to mechanical loading, is given by
\be\label{eq:Ffg}
\Ff_g(\bphi,\bu) = -\int_\O \gravmech \cdot (\bx + \bu) = -\int_\O \left[\left( \sum_\a \phi_\a \rho_\a + (1-\phi_r)\rho_{s,r} \right) \g + \gravmech_\mathrm{ext} \right] \cdot \left(\bm{x} + \bu\right).
\ee
Note that $\Ff_g$ is neither convex nor concave, but smooth as the sum of linear and quadratic terms.

The Helmholtz free energy is a convex yet non-smooth function of $\bX$ because of the constraints $\phi_\flat \leq \phi \leq \phi^\sharp$ incorporated in the definition of $\Ff_s$ and the constraint~\eqref{eq:constraint} corresponding to the term $\Ff_c$. As a consequence, its subdifferential is not single valued. In particular, $\bY = (\bp, \ov \bsig, z)$ belongs to $\p\Ff(\bX)$ if $\bp=(p_n,p_w)$ is given by \eqref{eq:1.p}\eqref{eq:1.hp}\&\eqref{eq:1.chi}, if $\ov \bsig = 2 \mu \, \beps + \lambda \Tr\beps \, \bI- b\,  \pi \bI = \bsig - b \, \pi \bI$ and if $z = M\theta - \pi$. In the previous expressions, the Coussy pressure $\pi$ can be interpreted as the Lagrange multiplier for the constraint $\bX \in \Kk$. 
For the sum of both the Helmholtz and the gravitational energy, the hydrostatic phase pressures $p_\a - \rho_\a\g \cdot (\bx + \bu)$, acting as fluid potential in~\eqref{eq:1.cons}, is derived as as an element of the subdifferential $\partial_{\phi_\a}(\Ff +\Ff_g)$.

Assume now that $t\mapsto \bX(t)$ satisfies the systems~\eqref{eq:1}--\eqref{eq:BC}, then 
\begin{subequations}
 \label{eq:dFf.1}
\be\label{eq:dFf.1a}
\frac{\d}{\d t} \Ff(\bX) = \int_\O \bigg\{
\sum_{\a \in \{n,w\}} p_\a \, \p_t \phi_\a + \ov \bsig : \p_t \beps + z \p_t \theta \bigg\},
\ee
and
\be\label{eq:dFf.1b}
\frac{\d}{\d t}\Ff_g(\bphi, \bu) = -\int_\O \gravmech \cdot \p_t \bu - \int_\O \sum_{\a \in \{n, w\}} \left(\rho_\a \g \cdot (\bx + \bu) \right) \p_t \phi_\a.
\ee
\end{subequations}
First, $z= 0$ as a direct consequence of~\eqref{eq:3.theta}, while $-\div \ov \bsig = \gravmech$ owing to~\eqref{eq:2}. Therefore, Stokes' theorem provides 
\[
\int_\O  \ov \bsig : \p_t \beps -\int_\O \gravmech \cdot \p_t \bu = - \int_{\p\O} \p_t\bu \cdot \ov \bsig \n \overset{\eqref{eq:BC.u}}{=} 0.
\]
Besides, it follows from~\eqref{eq:1.cons} together with Stokes' theorem and~\eqref{eq:BC} that 
\[
\int_\O \sum_{\a \in \{n,w\}} \left(p_\a - \rho_\a \g \cdot(\bx + \bu) \right) \, \p_t \phi_\a = -  \int_\O \sum_{\a \in \{n,w\}} \frac{s_\a}{\mu_\a} \left|\bbK(\phi)^{1/2} \grad \left( p_\a - \rho_a \g \cdot(\bx + \bu) \right) \right|^2 + \Sigma^D,
\]
with 
\[
\Sigma^D =  \int_{\Gamma^D} 
\sum_{\a \in \{n,w\}} \left(p_\a^D - \rho_\a \g \cdot \bx\right) \frac{s_\a}{\mu_\a} \bbK(\phi) \grad \left(  p_\a - \rho_\a \g \cdot (\bx + \bu) \right) \cdot \n
\]
being the work of the force imposed on the fluid at the level of the Dirichlet boundary condition~\eqref{eq:BC.phi.D} and hydrostatic phase pressures. 
Therefore, \eqref{eq:dFf.1} yields 
\be\label{eq:dFfdt}
\frac{\d}{\d t} \left(\Ff(\bX) +  \Ff_g(\bphi, \bu) \right) = -  \int_\O \sum_{\a \in \{n,w\}} \frac{s_\a}{\mu_\a} \left|\bbK(\phi)^{1/2} \grad \left( p_\a - \rho_a \g \cdot(\bx + \bu) \right) \right|^2 + \Sigma^D, 
\ee
The first contribution in the right-hand side of~\eqref{eq:dFfdt} is nonpositive since $s_\a\geq 0$ and 
$\bbK(\phi)$ is symmetric definite positive. It encodes the entropy production of the system, in terms of the hydrostatic phase pressures. Note that the evolution of the solid is assumed to be reversible, in the sense that 
no entropy is produced by the mechanical deformation of the porous matrix.

\subsection{Weak formulation and main result}\label{ssec:2.weak}
The analysis we propose in the next section for system~\eqref{eq:1}--\eqref{eq:BC} strongly builds on the stability estimate~\eqref{eq:dFfdt}, and therefore on some mathematical counterparts of the second principle of thermodynamics. More precisely, system \eqref{eq:1}--\eqref{eq:BC} can be interpreted as a non-autonomous generalized gradient flow (see for instance~\cite{Mie11, Pel-lecture}). 
The particular structure of the model under consideration yields several difficulties. A first one comes from the fact that the equations governing the fluid flow and the solid deformation are coupled through the strong constraint~\eqref{eq:constraint}. 
A second difficulty comes from the degeneracy of the dissipation term (mainly in the fluid phase pressure contribution)
\be\label{eq:Dd}
\Dd =  \int_\O \sum_{\a \in \{n,w\}} \frac{s_\a}{\mu_\a} \bbK(\phi) \grad p_\a \cdot \grad p_\a \geq 0. 
\ee
Besides the lack of dissipation for the solid part coming from the reversibility pointed out above, another degeneracy comes from the fact that 
the prefactor $s_\a$ in $\Dd$ vanishes in regions where only one fluid phase is present. As a consequence, only a weak control on the phase pressures can be deduced from the control of $\Dd$, as usual in the two-phase setting. It motivates the introduction of the Kirchhoff transform to carry out the mathematical analysis. 

Let $\xi, \psi: [0,1]\to\R$ be the continuous and increasing functions respectively defined by 
\be\label{eq:Kirchhoff.def}
\xi(s) = \int_0^s \sqrt{z(1-z)} \gamma''(z) \d z \quad \text{and} \quad \psi(s) = \int_0^s {z(1-z)} \gamma''(z) \d z, \qquad s \in [0,1], 
\ee
then one readily deduces from~\eqref{eq:1.hp} that 
\be\label{eq:Kirchhoff.1}
s_n \grad \hat p_n = \grad \psi(s_n) = \sqrt{s_ns_w} \grad \xi(s_n), \qquad s_w \grad \hat p_w = - \grad \psi(s_n), 
\ee
and that 
\be\label{eq:Kirchhoff.2}
s_n |\grad \hat p_n|^2 + s_w |\grad \hat p_w|^2 = |\grad \xi(s_n)|^2. 
\ee
Therefore, provided $\bbK(\phi) \geq K_\flat \bbI$ in the sense of symmetric matrices for some $K_\flat>0$, we have 
\[
\Dd \geq \frac{K_\flat}{\mu^\sharp} \int_\O \sum_{\a\in\{n,w\}} s_\a |\grad p_\a|^2 = \frac{K_\flat}{\mu^\sharp}
 \int_\O \left\{ |\grad \xi(s_n)|^2 + \left|\grad(\pi +  \chi)\right|^2 \right\}
\]
where we have set $\mu^\sharp = \max\{\mu_n,\mu_w\}$.

To properly define the notion of weak solution, we still have to introduce some relevant functional spaces. 
In what follows, we denote by $H^1(\O)$ the usual Sobolev space equipped with the norm 
\[
\| v \|_{H^1(\O)}^2 = \| v \|_{L^2(\O)}^2 + \| \grad v \|_{L^2(\O)^d}. 
\]
We also denote by
\[
V = \{ v \in H^1(\O) \; \text{s.t.}  \; v_{|_{\G^D}}= 0 \}, 
\]
equipped with $\|v\|_V = \| \grad v \|_{L^2(\O)^d}$ 
which defines a norm thanks to the Poincaré inequality, 
and by $V'$ its topological dual. The $d$-dimensional product space of $V$ is denoted by $V^d$. We also define the closed subspace $\bU$ of $V^d$ as
\[
\bU = \{\bu \in V^d \; \text{s.t.}\; \div \bu \in H^1(\O)\}, 
\]
equipped with the norm 
\[
{\|\bu\|}_{\bU}^2 = \| \bu\|^2_{V^d} + \|\grad (\div \bu)\|_{L^2(\O)^d}^2.
\]

The boundary condition $\bp^D=(p_n^D, p_w^D)$ is assumed to not depend on time, and to be the restriction of some function (still denoted by) $\bp^D \in H^1(\O)^2$.  Then we can reconstruct $s_n^D$ by setting 
\be\label{eq:snD}
s_n^D = \Ss(p_n^D - p_w^D), 
\ee
where, similarly to what was proposed in~\cite{BLS09, CGP09, BJS09}, the inverse capillary function ${(\gamma')}^{-1}$ is extended into a continuous function on the whole $\R$ by 
\be\label{eq:Ss}
\Ss(p) = \begin{cases}
0 & \text{if}\; p \leq \gamma'(0) \\
{(\gamma')}^{-1}(p) & \text{if}\; \gamma'(0) \leq p \leq \gamma'(1) \\
1 & \text{if}\; p \geq \gamma'(1). \\
\end{cases}
\ee
As $\xi\circ\Ss$ is $\frac12$-Lipschitz continuous, then $\xi(s_n^D)$ also belongs to $H^1(\O)$. We end up with the following definition of a weak solution. 
\begin{Def}\label{def:weak}
A set of functions $\left( \bphi, \bu, \theta, \chi, \pi\right)$ with 
\begin{itemize}
\item $\phi_\a \in L^\infty(\R_+\times \O; \R_+)$ for $\a\in \{n,w\}$ satisfying $\phi_\flat \leq \phi \leq \phi^\sharp$ a.e. in $\R_+\times \O$, such that $\xi(s_n)  - \xi(s_n^D)\in L^2_\text{loc}(\R_+; V)$, 
\item $\bu \in L^2_\text{loc}(\R_+; \bU) \cap L^\infty_\text{loc}(\R_+; V^d)$, 
\item $\theta, \pi, \chi  \in L^2_\text{loc}(\R_+; H^1(\O))$ with $\chi \in \ov \chi(\phi)$ a.e. in $\R_+ \times \O$,
\end{itemize}
is said to be a {\em global in time weak solution} to the problem~\eqref{eq:1}--\eqref{eq:BC} if \eqref{eq:3.theta} and \eqref{eq:constraint}
hold almost everywhere in $\R_+\times\O$, and if 
\be\label{eq:weak.u}
\iint_{(0,T) \times \O} \Big(2 \mu \, \beps(\bu): \beps(\bv) + \lambda (\div \bu)(\div\bv) \Big) =  \iint_{(0,T) \times \O}b\, \pi\, \div \bv + \iint_{(0,T) \times \O} \gravmech(\bphi) \cdot \bv
\ee
for all $\bv \in L^2((0,T);V^d)$, and if 
\be\label{eq:weak.n}
\int_{\R_+} \int_\O \phi_n \p_t v +\int_\O \phi_n^0 v(0,\cdot)
+ \int_{\R_+} \int_\O 
\frac1{\mu_n} \bbK(\phi) \left( \grad \psi(s_n) + s_n \grad \left( (\pi + \chi) - \rho_n \g \cdot (\bx + \bu)\right)\right) \cdot \grad v 
= 0,
\ee
and
\be\label{eq:weak.w}
\int_{\R_+} \int_\O \phi_w \p_t v +\int_\O \phi_w^0 v(0,\cdot)
+ \int_{\R_+} \int_\O 
\frac1{\mu_n} \bbK(\phi) \left( - \grad \psi(s_n) + s_w \grad \left((\pi + \chi) - \rho_w \g \cdot (\bx + \bu)\right)\right) \cdot \grad v 
= 0,
\ee
for all $v \in C^1(\R_+; V)$, such that there exists some $T>0$ such that $v(t,\cdot) = 0$ for all $t\geq T$.
\end{Def}

Our main result is the existence of such a global in time weak solution under the assumptions we list below. 
\begin{enumerate}[{(H}1)]
%
\item\label{H.Cte} The viscosities $\mu_n, \mu_w$ and the densities $\rho_n, \rho_w$ are positive constants, whereas $\g \in \R^3$ is constant. 
The Lamé coefficients $\lambda, \mu$ are positive constants, as well as the Biot modulus $M$. The Biot coefficient $b$ is also constant and belongs to $(0,1]$. The porosity at rest $\phi_r$ is assumed to have regularity $\phi_r\in H^1(\Omega)$ with values in $[\phi_\flat, \phi^\sharp]$ for some constants $0<\phi_\flat < \phi^\sharp < 1$.

\item\label{H.gamma} The function $\gamma \in C([0,1]; \R_+) \cap C^1([0,1))$ is strictly convex and increasing, with $s \mapsto \sqrt{1-s}\,\gamma''(s)$ 
belonging to $L^1(0,1)$. It is extended into a lower-semicontinuous convex function $\gamma: \R \to [0,+\infty]$ by setting $\gamma(s) = +\infty$ is $s\notin [0,1]$. 
\item\label{H.K} The intrinsic permeability function $\bbK$ belongs to $C([\phi_\flat, \phi^\sharp]; S_d^{++}(\R))$. In particular, 
there exist $K^\sharp \geq K_\flat>0$ such that $K_\flat \bbI \leq \bbK(\phi) \leq K^\sharp \bbI$ for all $\phi \in [\phi_\flat, \phi^\sharp]$ in the sense of 
symmetric definite matrices. 

\item\label{H.init} The initial fluid content $\bphi^0 = \left(\phi_n^0, \phi_w^0\right)$ belongs to $L^\infty(\O; \R_+^2)$ with 
$\phi_\flat \leq \phi_n^0 + \phi_w^0 = \phi^0 \leq \phi^\sharp$ almost everywhere in $\O$. Besides, the initial displacement $\bu^0 \in H^1(\O)^d$, with $\bu^0 \in V^d$, and the microscopic compression $\theta^0 \in L^2(\O)$ are the unique solution to the elliptic problem
\[
(\bu^0, \theta^0) = \underset{{(\bu,\theta)\; \text{s.t.}\; \eqref{eq:constraint}}}{\operatorname{argmin}} \int_\O \left\{\frac{M}2 \theta^2 + \mu \, \beps(\bu):\beps(\bu) + \frac{\lambda}2 |\div \bu |^2 -\gravmech(\bphi^0) \cdot \bu\right\}.
\]
It is characterized by
\[
\phi^0 - b \, \div \bu^0 - \theta^0 = \phi_r, 
\]
as well as 
\[
\int_\O \left\{ 2 \mu \, \beps(\bu^0) : \beps(\bv) + \lambda (\div \bu^0)(\div \bv) \right\} = \int_\O b \pi^0 \div \bv + \int_\O \gravmech (\bphi^0)\cdot \bv, \qquad \forall \bv \in V^d,
\]
with $\pi^0 = M \theta^0$. 

\item\label{H.Dirichlet} The boundary data $\bp^D = (p_n^D,p_w^D) \in W^{1,\infty}(\O)^2$ do not depend on time.

\item \label{H.Omega} 
The domain $\O$ is a bounded and connected open subset of $\R^d$. Its boundary is Lipschitz continuous and splits into $\p\O = \G^D \cup \G^N$ with $\G^D \cap \G^N = \emptyset$. We assume that $\G^D$ has positive $(d-1)$-dimensional Hausdorff (or Lebesgue) measure, yielding some Poincaré inequality, as well as Korn's inequality. Moreover, we assume that there exists $\ctel{cte:mecha}$
depending only on $\O$ such that the unique solutions $\bv_1,\ \bv_2 \in V^d$ to 
\begin{subequations}
\label{eq:Brenner} 
\begin{alignat}{4}
\label{eq:Brenner.0-l2} 
\div \left(\beps(\bv_1) + \tilde \lambda \, \div \bv_1\,  \bbI \right) &=  \bw_1 &\quad& \text{in}\; \O, 
\qquad 
\left(\beps(\bv_1) + \tilde \lambda \, \div \bv_1 \, \bbI\right) \cdot \n &&= \0  &\quad&  \text{on}\; \G^N,\\
\label{eq:Brenner.0-h1}
\div \left(\beps(\bv_2) + \tilde \lambda \, \div \bv_2\,  \bbI \right) &= \grad w_2 &\quad& \text{in}\; \O, 
\qquad 
\left(\beps(\bv_2) + \tilde \lambda \, \div \bv_2 \, \bbI\right) \cdot \n &&= w_2 \, \n &\quad&  \text{on}\; \G^N,
\end{alignat}
with $\tilde \lambda >0$ and $\bw_1 \in L^2(\O)^d$ and $w_2 \in H^1(\O)$ satisfy $\bv_1,\ \bv_2 \in \bU$ and 
\begin{align}
\label{eq:Brenner.1-l2}
\tilde \lambda \| \grad (\div \bv_1) \|_{L^2(\O)^{d}} &\leq \cter{cte:mecha} \|\bw_1\|_{L^2(\O)},\\
\label{eq:Brenner.1-h1}
\tilde \lambda \| \grad (\div \bv_2) \|_{L^2(\O)^{d}} &\leq \cter{cte:mecha} \|w_2\|_{H^1(\O)}.
\end{align}
\end{subequations}
\item\label{H.weak-coupling} We assume that the coefficients of the problem satisfy the following {\em weak coupling} condition:
\[
\lambda > M \, b^2 \, \cter{cte:mecha}.
\]
where $\cter{cte:mecha}$ is the constant appearing in Assumption~(H\ref{H.Omega}).

\item \label{H.gravity} The body force $\gravmech$ is of the form prescribed by~\eqref{eq:2.grav}. We assume moreover that the external body force satisfies $\gravmech_\mathrm{ext}\in L^2(\Omega)$ and the reference density of the rock $\rho_{s,r} \in L^\infty(\Omega)$ does not depend on time.
\end{enumerate}

The above assumptions deserve some comments. First (H\ref{H.Cte}) requires the domain to be homogeneous in space. The extension to the case of heterogeneous porous media would of course be of great interest.

Rather that prescribing the capillary pressure function, we prescribe its antiderivative $\gamma$ in (H\ref{H.gamma}), the interpretation of which in terms of energy being a cornerstone of~\cite{CGM15,CGM17}, see also Section~\ref{ssec:2.thermo}. The setting we study does not allow the capillary energy density function $\gamma$ to depend on the porosity $\phi$, as suggested in the seminal work of Leverett~\cite{Leverett41}. This choice has been made to stick to the framework of Coussy~\cite{Coussy04}. Extending our result to the case where $\gamma$ also depends on $\phi$ wouldn't lead to major difficulties provided $\Ff_f$ in~\eqref{eq:Helmholtz} remains convex. However, our framework already encompasses classical models from the literature, as for instance the Brooks-Corey model. In the later, $\gamma(s) \sim (1-s)^{1-\frac1{\lambda_\text{BC}}}$ satisfies (H\ref{H.gamma}) for $\lambda_\text{BC} >2$, corresponding to a so-called narrow pore size distribution. Note that Assumption (H\ref{H.gamma}) can be easily relaxed by assuming merely that 
$s \mapsto (1-s)\,\gamma''(s)$ belongs to $L^1(0,1)$, at the price of a slightly weaker regularity requirement in Definition~\ref{def:weak}, that is $\psi(s_n) - \psi(s_n^D) \in L^2_\text{loc}(\R;V)$ instead of $\xi(s_n) - \xi(s_n^D) \in L^2_\text{loc}(\R;V)$. The proof, which can be readily completed by passing to the limit in yet another step of regularization, is left to the reader. Under such a relaxed assumption, the full range $\lambda_\text{BC}>1$ of Brook-Corey exponents can be recovered.

Assumption (H\ref{H.K}) gives a generic framework for the dependance of the permeability with respect to the porosity. This framework encompasses the classical models by Kozeny~\cite{Kozeny27} and Carman~\cite{Carman37}, but also more recent models~\cite{Costa06, SRZRK19}.

As the mechanical response of the porous matrix is instantaneous, it is natural to require the initial data $\bu^0$ and $\theta^0$ to be at equilibrium with the fluid distribution, the later being of finite energy. This is the purpose of Assumption~(H\ref{H.init}).

Assumption (H\ref{H.Omega}) looks reasonable as it extends to the case of more general boundary conditions a results which is known to hold true for convex, polygonal two-dimensional domains $\O$, cf. Brenner and Sung~\cite[Section 2]{BS92} in the pure traction or pure displacement regimes. Here, it is here merely extended to the case of more general boundary conditions of mixed type. Note however that even in the simpler case of the Laplace equation, this may lead to geometrical constraints on the splitting of $\p\O$ into $\G^D$ and $\G^N$, see for instance~\cite[Section 6.2]{Grisvard89}.
Note also that the full $H^1(\O)$ norm appears in the right-hand side of~\eqref{eq:Brenner.1-h1} since a constant {$w_2$} in~\eqref{eq:Brenner.0-h1} possibly yields a non-constant solution {$\bv_2$} due to the boundary condition on $\G^N$. 

Assumption (H\ref{H.Dirichlet}) is there for the sake of simplicity. In the case of time varying boundary conditions on the pressures, suitable regularity assumptions are needed, as for instance in \cite{CNV21}.

Assumption~(H\ref{H.weak-coupling}) is known in the literature as a weak coupling condition. This regime is realistic in many applications with low Skempton coefficient~\cite{Coussy04,ulm2004concrete}, and a similar condition appears in papers on numerical methods (see for instance~\cite{altmann2022decoupling}) in which naive coupling strategies are employed, in opposition to the celebrated fixed stress and undrained split~\cite{kim2011stability,both2017robust} approaches which allow to overpass the weak coupling regime for the simulation of (possibly single phase) poromechanics.

Finally, Assumption~(H\ref{H.gravity}) comprises typical practical scenarios. External body forces $\gravmech_\mathrm{ext}$ with weaker regularity, e.g., associated to non-homogeneous traction boundary conditions, may be of interest for practical applications and thus also the analysis. For simpler presentation these are, however, not further discussed here.

\medskip
The following theorem is the main result of our paper. 
\begin{thm}\label{thm:main}
Under Assumptions (H\ref{H.init})--(H\ref{H.weak-coupling}), there exists (at least) a global in time weak solution to the problem~\eqref{eq:1}--\eqref{eq:BC} in the sense of Definition~\ref{def:weak}.
\end{thm}
Our proof for this theorem, to be detailed in what follows, relies on compactness arguments. We apply two successive regularizations of the problem. 
First, we soften the constraint~\eqref{eq:1.chi} into $\chi = G_\eps(\phi)$, $\eps>0$ where $G_\eps$ is a suitable regularization of the maximal monotone graph $\ov\chi$. A prototypical choice for $G_\eps$ is $\phi\mapsto\eps \log\frac{\phi - \phi_\flat}{\phi^\sharp- \phi}$ for $\eps>0$. The hard constraint~\eqref{eq:1.chi} is recovered when $\eps$ tends to $0$. We also regularize the mobilities in~\eqref{eq:1.cons} to remove the degeneracy in pure phase zone $\{s_\a = 0\}$: we replace $s_\a$ by $k_\eps(s_\a) = \max(\eps, s_\a)$ in~\eqref{eq:1.cons}. This modification 
allows to derive some control on the phase pressure thanks to the control of the entropy production~\eqref{eq:Dd}. 
In order to initiate the process, we establish the well-posedness of the elliptic problem consisting in one step of the backward Euler scheme with time step ${h}>0$ for the regularized problem with $\eps>0$. Yet another regularization is required to justify properly our calculations: we make use of a Faedo-Galerkin (spectral) method to 
rigorously establish the regularity of the solutions to the discrete problem. 
Then we recover a global in time weak solution by passing first to the limit $\eps\to0$, then $h\to0$. 

The proof strongly builds on the second principle of thermodynamics, in the sense that the main estimate used in our existence proof is the control of the production of the 
Helmholtz free energy sketched out in Section~\ref{ssec:2.thermo}, opening the way to possible extensions to more complex (but still thermodynamically consistent) physical settings.

\section{The discrete and regularized system}\label{sec:TimeDis} 

\subsection{Regularization}\label{ssec:reg}

Let $\left(G_\eps\right)_{\eps>0} \subset L^1(\phi_\flat,\phi^\sharp)$ be a family of smooth increasing and onto functions from $(\phi_\flat,\phi^\sharp)$ to $\R$ vanishing at $\frac{\phi_\flat + \phi^\sharp}2$ with
\be\label{eq:Gbeta.strict}
G_\eps'(y) \geq \eps, \qquad y \in \R.
\ee
and 
\be\label{eq:Gbeta.L1}
G_\eps \underset{\eps\to 0}\longrightarrow 0 \quad \text{in}\; L^1(\phi_\flat,\phi^\sharp).
\ee
Then one infers deduces from Dini's theorem that $G_\eps$ tends to $0$ uniformly on any compact of $(\phi_\flat,\phi^\sharp)$, and in particular that 
\be\label{eq:Gbeta.pointwise}
G_\eps(y) \underset{\eps\to 0}\longrightarrow 0 \quad \text{for all}\; y \in (\phi_\flat,\phi^\sharp).
\ee

We also set
\be\label{eq:kdelta} k_\eps(s) = \max(\eps,s) \quad  \text{for all}\; s\in\R.\ee
The regularization of the mobility $k_\eps$ makes the problem coercive but yields a difficulty that was originally hidden by the degeneracy in pure phase regions. 
As suggested by the extension~\eqref{eq:Ss} of $(\gamma')^{-1}$ outside of $[\gamma'(0),\gamma'(1)]$, 
and in close connection to what was proposed in~\cite{CP12, BCH13}, it becomes necessary to extend the capillary pressure function $\gamma':[0,1] \to \R_+$ into a maximum monotone graph. This amounts to define the monotone and anti-monotone  graphs $\ov p_n$, $\ov p_w$ by 
\be\label{eq:ovpa}
\ov p_n(s) = \begin{cases}
(-\infty, \gamma(0)+\gamma'(0)] & \text{if}\; s = 0, \\
\gamma(s)+(1-s)\gamma'(s) & \text{if}\; s \in (0,1], 
\end{cases}
\qquad 
\ov p_w(s) = \begin{cases}
\gamma(s)-s\gamma'(s) & \text{if}\; s \in [0,1), \\
(- \infty, \gamma(1)-\gamma'(1)] & \text{if}\; s =1. 
\end{cases}
\ee
Then $\ov p_n - \ov p_w = \Ss^{-1}$ in the sense of the maximal monotone graphs. 
We also regularize and then extend to the whole $\R$ the function $\gamma$ by defining 
\[
\gamma_\eps(s) = \begin{cases}
\gamma(0) + \int_0^s \gamma_\eps'(z) \d z & \text{if}\;s \in [0,1], \\
+\infty & \text{otherwise},
\end{cases}
\]
where 
\[
\gamma_\eps'(s) = \gamma'(0) + \int_0^s \gamma''_\eps(z) \d z
\quad 
\text{and} \quad 
\gamma_\eps''(s) 
= \min\left(\eps^{-1}, \max(\eps, \gamma''(s)\right), \quad s \in [0,1]. 
\]
We infer from the dominated convergence theorem that $\gamma_\eps$ converges uniformly towards $\gamma$ on 
$[0,1]$ as $\eps$ tends to $0$, and from Dini's theorem that the Lipschitz continuous function $\Ss_\eps$ defined by 
\be\label{eq:Ss.beta}
\Ss_\eps(p) = \begin{cases}
0 & \text{if}\; p \leq \gamma_\eps'(0), \\
{(\gamma_\eps')}^{-1}(p) & \text{if}\; \gamma_\eps'(0) \leq p \leq \gamma_\eps'(1), \\
1 & \text{if}\; p \geq \gamma_\eps'(1), \\
\end{cases}
\ee
converges uniformly towards $\Ss$. We also incorporate these regularizations into the graphs $\ov p_{n,\eps}$ and $\ov p_{w,\eps}$ which are defined by~\eqref{eq:ovpa} where $\gamma$ and $\gamma'$ have been replaced by $\gamma_\eps$ and $\gamma_\eps'$ respectively.

\begin{lem}\label{lem:Phi}
Given $\bp=(p_n,p_w) \in \R^2$ and $\pi \in \R$, then there exists a unique $\bphi = (\phi_n, \phi_w) \in \R_+^2$ with 
$\phi_\flat \leq \phi = \phi_n+\phi_w \leq \phi^\sharp$ such that  
\be\label{eq:Pregdis.p}
p_\a = \hat p_{\a}+ \pi + G_\eps(\phi) \quad \text{for some}\; \hat p_{\a} \in \ov p_{\a,\eps}(s_n)\; \text{with}\; s_n = \phi_n / \phi.
\ee
Moreover, the mapping $\bPhi_\eps: \bp - \pi \mapsto \bphi$ is Lipschitz continuous with Lipschitz constant possibly blowing up with $\eps^{-1}$.
\end{lem}
\begin{proof}
It follows from~\eqref{eq:Pregdis.p} and from the definition~\eqref{eq:ovpa} of the graphs $\ov p_{\a,\eps}$ that 
\[
s_n = \Ss_\eps(p_n - p_w) = \Ss_\eps(p_n - \pi - (p_w - \pi)), \quad s_w = 1-s_n. 
\]
Moreover, the relation
\[
G_\eps(\phi) = s_n p_n + s_w p_w - \gamma_\eps(s_n) - \pi = s_n (p_n -\pi) + s_w (p_w-\pi) - \gamma_\eps(s_n)
\]
uniquely determines $\phi$ since $G_\eps$ is an invertible function. Then we can reconstruct $\phi_\a = \phi s_\a$. 
The Lipschitz continuity of $\bPhi_\eps$ follows from the Lipschitz continuity of $\Ss_\eps$, $\gamma_\eps$ and $G_\eps^{-1}$.
\end{proof}

The function $\bPhi_\eps$ can be interpreted as the differential of a convex function. Before stating our next lemma, 
we introduce the convex function 
\[
\Gg_\eps:  \begin{cases} 
[\phi_\flat,\phi^\sharp] \to \R_+ \\
z \mapsto \int_{\frac{\phi_\flat+\phi^\sharp}2}^z G_\eps(a) \d a, 
\end{cases}
\]
as well as the convex and compact subset of $\R^2$ 
\[\Kk_\bphi = \{\bphi = (\phi_n, \phi_w) \in \R_+^2 \; | \; \phi_\flat \leq \phi_n +  \phi_w \leq \phi^\sharp \}.\] 
\begin{lem}\label{lem:Phi.2}
Define the convex function $F_\eps: \R^2 \to \R_+$ by  
\[
F_\eps(\bphi) =  \begin{cases}
\phi \gamma_\eps\left(s_n\right) + \Gg_\eps(\phi) \quad \text{if}\; \bphi \in \Kk_\bphi, \\
+ \infty \quad \text{otherwise}, 
\end{cases}
\]
where $\phi = \phi_n+\phi_w$ and $s_n = \phi_n/\phi$, then $F_\eps$ is convex. Moreover, for $(\bp, \pi) \in \R^2 \times \R$ and $\bphi = \bPhi_\eps(\bp- \pi)$, then 
\be\label{eq:DFbeta}
(p_n-\pi, p_w-\pi) = DF_\eps(\bphi).
\ee
\end{lem}
\begin{proof}
The function $F_\eps$ is continuous on $\Kk_\bphi$ and continuously differentiable on the interior $\overset{\circ}{\Kk_\bphi}$ of $\Kk_\bphi$. Let $\bphi \in \overset{\circ}{\Kk_\bphi}$, then 
\[
DF_\eps(\bphi) = \left(\ov p_{n, \eps}(s_n) + G_\eps(\phi), \ov p_{w, \eps}(s_n) + G_\eps(\phi) \right), 
\]
the sets $\ov p_{\a, \eps}(s_n)$ being assimilated to their single value since $0< s_n < 1$ as $\bphi \in  \overset{\circ}{\Kk_\bphi}$. In view of formula~\eqref{eq:Pregdis.p}, we deduce that \eqref{eq:DFbeta} holds true. 
Since $\gamma_\eps$ is increasing, since $\Gg_\eps(z) \geq 0$ and since $\phi \geq \phi_\flat$ for $\bphi \in \Kk_\bphi$, one gets the uniform lower bound
\[
F_\eps(\bphi) \geq \phi_\flat \gamma_\eps(0) =  \phi_\flat \gamma(0) \geq 0. 
\]
Finally, the function $\bphi \mapsto \phi \gamma_\eps(s_n)$ is $1$-homogeneous and (not strictly) convex on $\overset{\circ}{\Kk_\bphi}$, while $\bphi\mapsto \Gg_\eps(\phi)$ is the composition of the convex function $\Gg_\eps$ with the linear one $\bphi \mapsto \phi$, so it is convex too, as well as $F_\eps$. 
\end{proof}

\begin{rem}
In the line of Lemma~\ref{lem:Phi}, one can show that $F_\eps$ is uniformly convex for $\eps>0$. We leave to the reader the proof of this property which will not be used explicitly in what follows.
\end{rem}

\subsection{Faedo-Galerkin space discretization}\label{ssec:Galerkin}
Concerning the discretization w.r.t. space, we build on a Faedo-Galerkin approach. Since $V$ is separable, there exists a family $\left(v_k\right)_{k\geq0} \subset V$ such that, denoting by 
$\Vv_k = \operatorname{span}\{v_\ell, \; \ell \leq k \},$ then 
\[
\overline{\bigcup_{k\geq 1} \Vv_k}^V = V.
\]
Besides, we also introduce the complete orthonormal family  $\left(w_k\right)_{k\geq 0}$ of $L^2(\O)$ made of eigenvectors of the Laplace equation with Neumann boundary conditions, i.e. 
\be\label{eq:wk}
\int_\O \grad w_k \cdot \grad v = \lambda_k \int_\O w_k v, \qquad \forall v \in H^1(\O), 
\ee
with $\lambda_0 = 0$, $w_0 \equiv \sqrt{m_\O}$ (here and in what follows, $m_\O$ stands for the $d$-dimensional Lebesgue measure of $\O$), whereas  $\lambda_{k+1} \geq \lambda_k >0$ and $\|w_k\|_{L^2(\O)} = 1$ for $k\geq 1$ as well. 
We denote by 
$\Ww_k = \operatorname{span}\{w_\ell,\; \ell \leq k\} \subset H^1(\O),$ 
and classical results from the spectral theory of self-adjoint compact operators (see for instance~\cite{Brezis11}) show that 
\be\label{eq:Wk.dense}
\overline{\bigcup_{k\geq 1} \Ww_k}^{L^2(\O)} = L^2(\O).
\ee
The (topological) dual $\Ww_k'$ of $\Ww_k$ is identified to $\Ww_k$ thanks to the Riesz theorem building on the $L^2(\O)$-scalar product.
Eventually, we denote by $\ov \bu_k \in V^d$ the unique solution to 
\be\label{eq:ovuk}
\int_\O \left[2 \mu \, \beps(\ov \bu_k) + \lambda \div \ov \bu_k \right] : \beps(\bv) = b \int_\O w_k \div \bv, \qquad \forall \bv \in V^d, \; k \geq 0.
\ee
where $w_k$ is the $k^\text{th}$ eigenvector of the Laplace operator introduced above in~\eqref{eq:wk}. 
We denote by 
\[
\bUu_k = \operatorname{span}\left\{ \ov \bu_\ell, \; 1 \leq \ell \leq k \right\}, 
\]
then $\bUu_k \subset \bU$ thanks to Assumption~(H\ref{H.Omega}). 

\subsection{The discrete and regularized problem with frozen mobility and linearized gravity}

Since $\Kk_\bphi$ is convex in $\R^2$, the orthogonal projection 
\[\bPi: \begin{cases}
\R^2 \to \Kk_\bphi \\
\bphi = (\phi_n, \phi_w) \mapsto \bPi(\bphi) = (\Pi_n(\bphi), \Pi_w(\bphi))
\end{cases}
\] is uniquely defined. 
For $\tilde \bphi = (\tilde \phi_n , \tilde \phi_w) \in L^2(\O)^2$, we define 
\be\label{eq:tilde.sphi}\text{$\tilde \phi = \Pi_n(\tilde \bphi) + \Pi_w(\tilde \bphi)$ \quad and \quad 
$\tilde s_\a = \Pi_\a(\tilde \bphi) / \tilde \phi.$}\ee
Note that these definitions are consistent with the previous ones in the case where $\tilde \bphi \in \Kk_\bphi$.

\begin{prop}\label{prop:Pregdis}
Let 
$\bphi^\star = (\phi_n^\star, \phi_w^\star) \in L^\infty(\O)^2$ with $\bphi^\star \in \Kk_\bphi$ a.e. in $\O$, let $\tilde \bphi = (\tilde \phi_n , \tilde \phi_w) \in L^2(\O)^2$, let $\tilde\bu \in \bU$, and given $\eps>0$ and $h>0$, then for any $k\geq 1$,  there exists a unique $\left(\bphi_k, \bp_k,  \bu_k, \theta_k, \pi_k\right)$ such that $\bp_k= (p_{n,k}, p_{w,k})$ belongs to $\bp^D + \Vv_k$, such that  $\pi_k$ and $\theta_k$ belong to $\Ww^k$ with 
$
\pi_k = M \theta_k, 
$
such that $\bu_k = \hat \bu + \bu_k^{o}$ with $\hat\bu \in \bU$ and $\bu_k^o \in \bUu_k$ respectively solving
\begin{subequations}
\label{eq:discretization.k}
\begin{alignat}{5}\label{eq:mecha.k.1}
\int_\O \hat \bsig: \beps(\bv) &= \int_\O \gravmech(\tilde \bphi) \cdot \bv, 
&\quad& \forall \bv \in V^d, \quad \text{with}\; &\hat \bsig &= 2 \mu \, \beps(\hat \bu) + \lambda \div \hat \bu \bbI, \\
\label{eq:mecha.k.2}
\int_\O \bsig^o_k: \beps(\bv) &= \int_\O b \,\pi_k \div \bv, 
&\quad& \forall \bv \in \bUu_k, \quad \text{with}\ \  &\bsig_k^o &= 2 \mu \, \beps(\bu_k^o) + \lambda \div \bu_k^o \bbI, 
\end{alignat}
such that $\bphi_k = (\phi_{n,k}, \phi_{w,k}) = \bPhi_\eps(\bp_k- \pi_k)$ satisfies
\be\label{eq:constraint.k}
\int_\O (\phi_k - b \div \bu_k - \theta_k) w = \int_\O\phi_r w,\qquad \forall w \in \Ww_k,
\ee
and such that
\be\label{eq:Pregdis.cons.k}
\int_\O \frac{\phi_{\a,k} - \phi_\a^\star}{h} v + \int_\O \frac{k_\eps(\tilde s_\a)}{\mu_\a} \bbK(\tilde \phi) \grad \left( p_{\a,k} - \rho_\a \g\cdot (\bx + \tilde{\bu})\right) \cdot \grad v = 0 \quad \text{for all}\; v \in \Vv_k, \; \a \in \{n,w\}.
\ee
\end{subequations}
\end{prop}


 \begin{rem}\label{rem:mecha.k}
 As a result of the particular choice~\eqref{eq:ovuk} for the basis functions $\left(\ov \bu_\ell\right)_{1 \leq \ell \leq k}$ of $\Uu_k$, the relation~\eqref{eq:mecha.k.2} 
 holds true for all $\bv \in V^d$ and not only for $\bv \in \bUu_k$. Indeed, since $\pi_k$ belongs to $\Ww_k$, it can be written decomposed into
 $
 \pi_k = \sum_{\ell = 0}^k \pi_{k,\ell} w_\ell
 $
 with $w_\ell$ fulfilling~\eqref{eq:wk}. Then we deduce from~\eqref{eq:ovuk} that $\bu_k^o= \sum_{\ell = 0}^k \pi_{k,\ell} \ov \bu_\ell$ satisfies
 \be\label{eq:2}
 \int_\O \bsig_k^o: \beps(\bv) = \int_\O b \, \pi_k \div \bv, \qquad \forall \bv \in V^d. 
 \ee
 In other words, $\bu_k^o$ is the genuine continuous solution to the linear mechanics system corresponding to the approximate right-hand side $b \grad \pi_k$, and thus so does $\bu_k = \hat \bu + \bu_k^o$ for the full body force term $b \grad \pi_k + \gravmech(\tilde \bphi)$, i.e.
 \be\label{eq:mecha.k}
 \int_\O \bsig_k:\beps(\bv) = \int_\O \left( \gravmech(\tilde \bphi)\cdot \bv + b \pi_k \div \bv\right), \quad \forall \bv \in V^d, \quad \text{with}\; 
 \bsig_k = 2 \mu\, \beps(\bu_k) + \lambda \, \div \bu_k \, \bbI.
 \ee
 Then owing to Assumption~(H\ref{H.Omega}) and~(H\ref{H.gravity}), for all $k\geq 0$ the triangle inequality yields
\be
\label{eq:D2uk}
\| \grad(\div \bu_k)  \|_{L^2(\O)^{d}} 
\leq 
\frac{\cter{cte:mecha}}\lambda \| \gravmech(\tilde \bphi) \|_{L^2(\O)^d}
+
\frac{b\, \cter{cte:mecha}}\lambda \| \pi_k\|_{H^1(\O)}.
\ee
 \end{rem}

\begin{proof}[Proof of Proposition~\ref{prop:Pregdis}]
Define 
\[
\bHh_k: \begin{cases}
(\Vv_k)^2 \times \bUu_k \times \Ww_k \times \Ww_k \to (\Vv_k')^2 \times \bUu_k' \times \Ww_k'\times \Ww_k' \\
(\bp_k^{o} = (p^{o}_{n,k}, p^{o}_{w,k}), \bu_k^o, \theta_k, \pi_k) \mapsto \left( \left(r_{n,k},r_{w,k}\right), \br_{\bu,k}, r_{\theta,k}, r_{\pi,k} \right)
\end{cases}
\]
by setting $\bphi_k=  \bPhi_\eps(\bp_k^o + \bp^D- \pi_k)$, $p_{\a,k} = p_{\a,k}^o + p_\a^D$, $\bu_k = \hat \bu + \bu^o_k$ and
\begin{align*}
\langle r_{\a,k}, v \rangle_{\Vv_k',\Vv_k} =&\;  \int_\O (\phi_{\a,k} - \phi_\a^\star) v + {h} \int_\O \frac{k_\eps(\tilde s_\a)}{\mu_\a} \bbK(\tilde \phi) \grad (p_{\a,k} - \rho_\a \g\cdot (\bx + \tilde{\bu})) \cdot \grad v, 
\quad \forall v \in \Vv_k, 
\\
\langle \br_{\bu,k}, \bv \rangle_{\bUu_k',\bUu_k} =&\; \int_\O \left\{ 2 \mu \, \beps(\bu_k^o) : \beps(\bv) + \lambda (\div\bu_k^o)(\div\bv) - b\, \pi_k\, \div\bv \right\}, \, \qquad \forall \bv \in \bUu_k,
\end{align*}
and 
\[
r_{\theta,k} = 2(M\theta_k - \pi_k), \qquad r_{\pi,k} = - \phi_k + \phi_r+ \theta_k + b \div \bu_k^o + b \div \hat \bu + \theta_k - \frac{\pi_k}M.
\]
Let $\bY_k^{(i)} = \left(\bp_k^{o,(i)} = (p_{n,k}^{o,(i)}, p_{w,k}^{o,(i)}), \bu_k^{o,(i)}, \theta_k^{(i)}, \pi_k^{(i)}\right)$, $i=1,2$, be two elements of 
$(\Vv_k)^2 \times \bUu_k \times \Ww_k \times \Ww_k $, 
and denote by $ \bR_k^{(i)} = \left( (r_{n,k}^{(i)},r_{w,k}^{(i)}), \br_{\bu,k}^{(i)}, r_{\theta,k}^{(i)}, r_{\pi,k}^{(i)} \right) = \bHh\left(\bY_k^{(i)}\right)$, then 
one checks that 
\begin{align*}
\label{eq:monotone} 
\left\langle \bR_k^{(1)} - \bR_k^{(2)}\, , \, \bY_k^{(1)} - \bY_k^{(2)} \right\rangle= 
&
 \sum_{\a\in\{n,w\}} \int_\O (\phi_{\a,k}^{(1)}-\phi_{\a,k}^{(2)})(p_{\a,k}^{(1)} - \pi_k^{(1)} - p_{\a,k}^{(2)} + \pi_k^{(2)}) 
\\
\nn
&
+ {h}  \sum_{\a\in\{n,w\}} \int_\O \frac{k_\eps(\tilde s_\a)}{\mu_\a} \bbK(\tilde \phi) \grad (p_{\a,k}^{(1)} - p_{\a,k}^{(2)}) \cdot  \grad (p_{\a,k}^{(1)} - p_{\a,k}^{(2)})
 \\
 \nn
 &
 + \int_\O \left\{ 2 \mu \, \beps(\bu_k^{(1)} - \bu_k^{(2)}) :  \beps(\bu_k^{(1)} - \bu_k^{(2)}) + \lambda \left(\div (\bu_k^{(1)} - \bu_k^{(2)}) \right)^2  \right\} 
 \\
 \nn
 &
 + \int_\O \left\{M (\theta_k^{(1)}-\theta_k^{(2)})^2 + \frac1M\left( M(\theta_k^{(1)} - \theta_k^{(2)}) - \pi_k^{(1)}+\pi_k^{(2)} \right)^2\right\}.
\end{align*}
with gravity contributions canceling, as well as those related to $\hat \bu$. 
Owing to the convexity of $\bPhi_\eps$ established in Lemma~\ref{lem:Phi.2}, one has
\[
 \sum_{\a\in\{n,w\}} \int_\O (\phi_{\a,k}^{(1)}-\phi_{\a,k}^{(2)})(p_{\a,k}^{(1)} - \pi_k^{(1)} - p_{\a,k}^{(2)} + \pi_k^{(2)})  \geq0.
 \]
Since  $\bbK(\tilde \phi) \geq K_\flat \bbI$ and $\mu_\a \leq \mu^\sharp$, one gets that 
\[
 \sum_{\a\in\{n,w\}} \int_\O \frac{k_\eps(\tilde s_\a)}{\mu_\a} \bbK(\tilde \phi) \grad (p_{\a,k}^{(1)} - p_{\a,k}^{(2)}) \cdot  \grad (p_{\a,k}^{(1)} - p_{\a,k}^{(2)})
\geq  \eps \frac{K_\flat}{\mu^\sharp}  \big\|\bp_k^{(1)} - \bp_k^{(2)} \big\|^2_{V^2}.
\]
Note that $\Vv_k$ is equipped with the norm $\| v \|_V = \|\grad v \|_{(L^2)^d}$ which is a norm since we assumed that $\G^D$ has positive measure thanks to Poincaré inequality. We also deduce from Korn inequality that there exists $\ctel{cte:Korn}$ depending only on $\O$ and $\G^D$ (but not on $k$) such that 
\[
 \int_\O \left\{ 2 \mu \, \beps(\bu_k^{(1)} - \bu_k^{(2)}) :  \beps(\bu_k^{(1)} - \bu_k^{(2)}) + \lambda \left(\div (\bu_k^{(1)} - \bu_k^{(2)}) \right)^2  \right\} 
 \geq \cter{cte:Korn} \mu \| \bu_k^{(1)} - \bu_k^{(2)} \|_{V^d}^2. 
\]
Further, elementary calculations show that 
\[
x^2 + (x - y)^2 \geq \frac{3-\sqrt5}2 (x^2 + y^2), \qquad \forall x,y \in \R, 
\]
so that 
\begin{multline*}
 \int_\O \left\{M (\theta_k^{(1)}-\theta_k^{(2)})^2 + \frac1M\left( M(\theta_k^{(1)} - \theta_k^{(2)}) - \pi_k^{(1)}+\pi_k^{(2)} \right)^2\right\}
\\ \geq \frac{3-\sqrt5}{2}\left( M \| \theta_k^{(1)} - \theta_k^{(2)}\|_{L^2(\O)} + \frac1M  \| \pi_k^{(1)} - \pi_k^{(2)}\|_{L^2(\O)} \right).
\end{multline*}
Then we infer from previous estimates that there exists~$\ctel{cte:monotone}>0$ depending on the data of the continuous problem as well as on $\eps$ and ${h}$ (but neither on $k$ nor on $\eps$) such that 
\begin{multline*}
\left\langle \bR_k^{(1)} - \bR_k^{(2)}\, , \, \bY_k^{(1)} - \bY_k^{(2)} \right\rangle \\\geq \cter{cte:monotone} \Big( \|\bp_k^{(1)} - \bp_k^{(2)}\|_{V^2}^2 +  \|\bu_k^{(1)} - \bu_k^{(2)}\|_{V^d}^2 
+ \| \theta_k^{(1)} - \theta_k^{(2)} \|_{L^2(\O)}^2 +  \| \pi_k^{(1)} - \pi_k^{(2)} \|_{L^2(\O)}^2\Big).
\end{multline*}
As a consequence, $\bHh_k$ is strongly monotone, whence there exists a unique $\bY_k = (\bp_k^o, \bu_k^o, \theta_k, \pi_k)$ such that $\bHh_k(\bY_k) = \0$ owing to~\cite[Corollaire 17]{Brezis6566} (see also \cite{Brezis73, Lions69}). This also allows to reconstruct $\bphi_k = \bPhi_\eps(\bp_k - \pi_k)$.
\end{proof}

The above proof (and thus the statement of Proposition~\ref{prop:Pregdis}) is still valid at the limit $k\to +\infty$, leading to the well-posedness of the 
limiting problem in the Hilbert space $V^2 \times V^d \times L^2(\O) \times L^2(\O)$. The regularity provided by Proposition~\ref{prop:Pregdis} is however not sufficient to carry out our mathematical study and to pass to the limit to recover the continuous problem. This was the motivation for the space discretization, the regularizing effect of which being needed to establish rigorously the following proposition. 

\begin{prop}\label{prop:Pregdis.NRG}
Define the approximate Helmholtz free energy $\Ff_\eps$ by 
\[\Ff_\eps(\bX) = \int_\O \left(F_\eps(\bphi) + \mu \, \beps(\bu):\beps(\bu) + \frac\lambda2 |\div\bu|^2 + \frac M2|\theta|^2\right) 
\quad \text{for}\; \bX = (\bphi, \bu, \theta, \pi).\]
Let $\bphi^\star \in L^\infty(\O; \R_+^2)$ with $\bphi^\star \in \Kk_\bphi$ a.e. in $\O$, and let $\bu^\star \in \bU$ and $\pi^\star = M \theta^\star$ in $L^2(\O)$ be such that 
\be\label{eq:constraint.star}
\phi^\star - b \div \bu^\star - \theta^\star = \phi_r, \qquad \text{with}\; \phi^\star = \phi_n^\star + \phi_w^\star. 
\ee
Denoting by $\bX^\star = (\bphi^\star, \bu^\star, \theta^\star, \pi^\star)$, then for any $\tilde \bphi \in L^2(\O)^2$ and $\tilde \bu \in H^1(\Omega)^d$, the unique solution $\bX_k = (\bphi_k, \bu_k, \theta_k, \pi_k)$ to the discrete regularized problem, cf. Proposition~\ref{prop:Pregdis}, 
satisfies 
\begin{multline}
\left( 1 - \cter{cte:Pregdis.NRG.2} h \eps^2 \right) \Ff_\eps(\bX_k) \\
+ {h}\eps \cter{cte:Pregdis.NRG}  \left( \sum_{\a \in\{n,w\}}  \left\|\grad p_{\a,k} \right\|_{L^2(\O)^d}^2 + \left\|\grad (G_\eps(\phi_k) + \pi_k) \right\|_{L^2(\O)^d}^2+ \eps \left\|\pi_k \right\|_{H^1(\O)}^2\right) \\\leq \Ff_\eps(\bX^\star) +  \int_\O \gravmech(\tilde \bphi) \cdot (\bu_k - \bu^\star) \\+ \sum_{\a \in \{n,w\}} \int_\O (\phi_{\a,k} - \phi_\a^\star) p_\a^D + \cter{cte:NRG.1} {h}  
(1 + \|\tilde \bu\|_{V^d}^2 + \|\tilde \bphi \|_{L^2(\O)^2}^2),
\label{eq:Pregdis.NRG}
\end{multline}
for positive constants~$\ctel{cte:Pregdis.NRG}$, $\ctel{cte:Pregdis.NRG.2}$ and $\ctel{cte:NRG.1}$ which neither depend on $k,\eps$ and ${h}$, nor on $\tilde \bphi$ or $\tilde \bu$. 
Moreover, for $\eps$ small enough to ensure that $\cter{cte:Pregdis.NRG.2} h \eps^2\leq 1$,  there exists $\ctel{cte:chi}$ depending neither on $k$ nor on $\tilde \bphi$ (but possibly on $\eps$ and $h$) such that 
\be\label{eq:NRG.chi.k}
\| \pi_k\|_{H^1(\O)} + \|G_\eps(\phi_k) \|_{H^1(\O)} \leq  \cter{cte:chi}. 
\ee
\end{prop}

\begin{proof}
Testing~\eqref{eq:Pregdis.cons.k} by ${h}\, (p_{\a,k}- p_\a^D) \in \Vv_k$ and summing over $\alpha \in \{n,w\}$ provides 
\[
\Aa_k + \Bb_k + \Dd_k = \Rr_{1,k} + \Rr_{2,k} + \Rr_{3}, 
\]
where 
\begin{align*}
\Aa_k =& \sum_{\a\in\{n,w\}}\int_\O (\phi_{\a,k} - \phi_\a^\star) (p_{\a,k} - \pi_{a,k}), 
\\
\Bb_k =& \sum_{\a\in\{n,w\}}\int_\O (\phi_{\a,k} - \phi_\a^\star) \pi_{\a,k}, 
\\
\Dd_k = & \;h \sum_{\a \in \{n,w\}} \int_\O \frac{k_\eps(\tilde s_\a)}{\mu_\a} \bbK(\tilde \phi) \grad p_{\a,k} \cdot \grad p_{\a,k}, 
\\
\Rr_{1,k} = & \sum_{\a \in \{n,w\}} \int_\O (\phi_{\a,k} - \phi_\a^\star) p_\a^D,
\\
\Rr_{2,k} = &   \;h \sum_{\a \in \{n,w\}}  \int_\O \frac{k_\eps(\tilde s_\a)}{\mu_\a} \bbK(\tilde \phi) \grad p_{\a,k} \cdot \grad \left(\rho_\a \g\cdot (\bx + \tilde{\bu}) + p_\a^D\right),
\\
\Rr_{3,k} = &   - h\sum_{\a \in \{n,w\}}  \int_\O \frac{k_\eps(\tilde s_\a)}{\mu_\a} \bbK(\tilde \phi) \grad p_\a^D \cdot \grad\left( \rho_\a\g \cdot (\bx + \tilde \bu) \right).
\end{align*}
Young's inequality implies that 
\[
\Rr_{2,k} \leq \frac12 \Dd_k + \frac12 h \sum_{\a \in \{n,w\}}  \int_\O\frac{k_\eps(\tilde s_\a)}{\mu_\a} \left|\bbK(\tilde \phi)^{1/2} \grad \left(\rho_\a \g\cdot (\bx + \tilde \bu) + p_\a^D\right)\right|^2, 
\]
so that 
\be\label{eq:ADRk}
\Aa_k + \Bb_k + \frac12 \Dd_k \leq \Rr_{1,k} + \cter{cte:NRG.1} {h} (1 + \| \tilde \bu \|_{V^d}^2)
\ee
for some $\cter{cte:NRG.1}$ independent on $k$,${h}$, and $\epsilon$.
As a consequence of Lemma~\ref{lem:Phi.2}, we deduce from a convexity inequality that 
\[
\Aa_k =  \sum_{\a \in \{n,w\}} \int_\O (\phi_{\a,k}- \phi_\a^\star)(p_{\a,k} - \pi_k) \geq \int_\O (F_\eps(\bphi_k) - F_\eps(\bphi^\star)).
\]
On the other hand, we infer from~\eqref{eq:constraint.k} and \eqref{eq:constraint.star} that 
\[
\Bb_k = \int_\O (\phi_{k}- \phi^\star) \pi_k= \int_\O (b \div (\bu_{k} - \bu^\star) +  \theta_k-\theta^\star) \pi_k.
\]
Since $\pi_k = M\theta_k$, the elementary convexity inequality $(x-y)x \geq \frac12(x^2-y^2)$ provides
\[
\int_\O (\theta_k-\theta^\star) \pi_k \geq \frac{M}2\left( \|\theta_k\|_{L^2(\O)}^2 -  \|\theta^*\|_{L^2(\O)}^2\right). 
\]
Besides, employing the same inequality again, we deduce from~\eqref{eq:mecha.k} that 
\begin{align*}
 \int_\O b \div (\bu_{k} - \bu^\star) \pi_k =& \int_\O \bsig_k : \beps(\bu_{k} - \bu^\star)- \int_\O \gravmech(\tilde \bphi)\cdot(\bu_k - \bu^\star) \\
 \geq& \int_\O \mu \left(\beps(\bu_k):\beps(\bu_k) - \beps(\bu^\star):\beps(\bu^\star)\right) \\
 & \qquad + \frac\lambda2 
\left(| \div\bu_k|^2- | \div\bu^\star|^2\right)
 - \int_\O \gravmech(\tilde \bphi)\cdot(\bu_k - \bu^\star).
\end{align*}
Collecting the above estimates, we get that 
\be\label{eq:Ak}
\Aa_k + \Bb_k \geq \Ff_\eps(\bX_k) - \Ff_\eps(\bX^\star)- \int_\O \gravmech(\tilde \bphi)\cdot(\bu_k - \bu^\star).
\ee
Due to the assumptions on $\bphi^\star$ and to the regularity of $\bu^\star$ and $\theta^\star$, the regularized Helmholtz free energy $\Ff_\eps(\bX^\star)$ 
of $\bX^\star = (\bphi^\star, \bu^\star, \theta^\star, M\theta^\star)$ is finite. In the end, the gravity-related contribution is linear in $\bu_k$ which will allow its control.
\medskip

On the other hand, since $k_\eps(\tilde s_\a) \geq \eps$, since $\bbK(\tilde \phi) \geq K_\flat \bbI$ and since $0 \leq s_{\a,k} \leq 1$, one gets that 
\begin{align}
\Dd_k \geq& {h} \eps\frac{K_\flat}{\mu^\sharp} \sum_{\a \in \{n,w\}} \int_\O |\grad p_{\a,k}|^2 \nonumber\\
\geq & {h}  \eps \frac{K_\flat}{2 \mu^\sharp} \sum_{\a \in \{n,w\}} \left( \int_\O |\grad p_{\a,k}|^2 + \int_\O s_{\a,k} |\grad p_{\a,k}|^2 \right) =:\Dd_{1,k} + \Dd_{2,k}. 
\label{eq:piH1.-1}
\end{align}
The first term in the right-hand side provides some control on $\|p_{\a,k}\|_{V}$. Concerning the second contribution, the relation $\sum_{\a \in \{n,w\}} s_{\a,k} \grad \hat p_{\a,k} = \0$, cf.~\eqref{eq:Kirchhoff.1}, allows to reformulate 
\begin{align}
\sum_{\a \in\{n,w\}} s_{\a,k} |\grad p_{\a,k}|^2 =& \sum_{\a \in\{n,w\}} s_{\a,k} |\grad (\hat p_{\a,k} + G_\eps(\phi_k) + \pi_k)|^2 \nonumber\\
=&  \sum_{\a \in\{n,w\}} s_{\a,k} |\grad \hat p_{\a,k}|^2  + |\grad (G_\eps(\phi_k) + \pi_k)|^2 \geq  |\grad (G_\eps(\phi_k) + \pi_k)|^2.  \label{eq:piH1.0}
\end{align}
Introducing the non-decreasing function $\wt G_\eps: z \mapsto G_\eps(z) - \eps\, z$, 
the latter term rewrites 
\begin{align*}
 |\grad (G_\eps(\phi_k) + \pi_k)|^2 =\;&  |\grad (\wt G_\eps(\phi_k) + \eps \phi_k + \pi_k)|^2 \\
 =\; &  |\grad (\wt G_\eps(\phi_k)+ \pi_k)|^2 + \eps^2 |\grad \phi_k|^2 + 2 \eps \grad \phi_k \cdot \grad \pi_k + 2 \grad \wt G_\eps(\phi_k) \cdot \grad \phi_k.
\end{align*}
We deduce from the monotonicity of $ \wt G_\eps$ that $\grad \wt G_\eps(\phi_k) \cdot \grad \phi_k \geq 0$, hence 
\be\label{eq:piH1.1}
|\grad (G_\eps(\phi_k) + \pi_k)|^2  \geq 
2 \eps  \grad \phi_k \cdot \grad \pi_k. 
\ee

The particular choice for the space $\Ww_k$ will be used here. Let $w_\ell$ be such that  \eqref{eq:wk} holds true, then since $\phi_k\in H^1(\Omega)$, following from the Lipschitz continuity of $\bPhi_\eps$, it holds
\[
\int_\O \grad \phi_k \cdot \grad w_\ell = \lambda_\ell \int_\O \phi_k\, w_\ell, \qquad 0 \leq \ell \leq k.
\]
Thus, exploiting the fact that $\phi_r \in H^1(\O)$ thanks to (H\ref{H.Cte}) and $\div \bu_k \in H^1(\O)$ thanks to (H\ref{H.Omega}) and more precisely to~\eqref{eq:D2uk}, we can infer that for $0 \leq \ell \leq k$,
\[
\int_\O \grad \phi_k \cdot \grad w_\ell \overset{\eqref{eq:wk}}= \lambda_\ell \int_\O \phi_k \, w_\ell \overset{\eqref{eq:constraint.k}}= 
\lambda_\ell \int_\O  (\phi_r + b\,  \div \bu_k + \theta_k)w_\ell \overset{\eqref{eq:wk}}= \int_\O  \grad (\phi_r + b \div \bu_k + \theta_k)\cdot \grad w_\ell.
\]
Therefore, 
\[
\int_\O \grad \phi_k \cdot \grad w = \int_\O  \grad (\phi_r + b \div \bu_k + \theta_k)\cdot \grad w, \qquad \forall w \in \Ww_k.
\]
In particular for $w=\pi_k$, and bearing in mind that $M\theta_k = \pi_k$, this yields
\[
\int_\O \grad \phi_k \cdot \grad \pi_k = \int_\O \left(\frac1M |\grad \pi_k|^2 + \grad (\div \bu_k) \cdot b \grad \pi_k + \grad \phi_r \cdot \grad\pi_k \right).
\]
It follows from \eqref{eq:D2uk} combined with Cauchy-Schwarz that 
\[
\int_\O  \grad (\div \clemrev{\bu_k}) \cdot \grad \pi_k \geq
 - \frac{b\,\cter{cte:mecha}}\lambda \|\pi_k\|_{H^1(\O)}^2 \clemrev{- \frac{\cter{cte:mecha}}{\lambda}\| \gravmech(\tilde \bphi) \|_{L^2(\O)^d} \|\pi_k\|_{H^1(\O)},} \]
\clemrev{so that Young's inequality provide
\begin{multline*}
\int_\O  \grad (\div \clemrev{\bu_k}) \cdot \grad \pi_k \geq -  \frac{b\,\cter{cte:mecha}}\lambda \|\pi_k\|_{H^1(\O)}^2 
\\
- \frac{1}{4b} \left(\frac1M -  \frac{b^2\cter{cte:mecha}}{\lambda}\right) 
\|\pi_k\|^2_{H^1(\O)} -  \left(\frac1M -  \frac{b^2\cter{cte:mecha}}{\lambda}\right)^{-1} b\, \|\gravmech(\tilde \bphi) \|_{L^2(\Omega)^d}^2.
\end{multline*}
Similarly, we obtain that 
}
\[
 \int_\O \grad \phi_r \cdot \grad\pi_k \geq - \frac{1}{\clemrev 4} \left(\frac1M -  \frac{b^2\cter{cte:mecha}}{\lambda}\right)
\|\pi_k\|^2_{H^1(\O)} -\left(\frac1M -  \frac{b^2\cter{cte:mecha}}{\lambda}\right)^{-1} \|\grad \phi_r \|_{L^2(\Omega)^d}^2.
\]
Therefore, using again $\pi_k = M \theta_k$, we obtain that 
\begin{align}\label{eq:gradphi-gradpi.k}
\int_\O \grad \phi_k \cdot \grad \pi_k \geq &\; \frac12 \left(\frac1M -  \frac{b^2\cter{cte:mecha}}{\lambda}\right)
\|\pi_k\|^2_{H^1(\O)} - M \| \theta_k \|_{L^2(\O)}^2  \\ 
& \qquad \nonumber
\clemrev{- \left(\frac1M -  \frac{b^2\cter{cte:mecha}}{\lambda}\right)^{-1} \left( \|\grad \phi_r \|_{L^2(\Omega)^d}^2 + b^2  \|\gravmech(\tilde \bphi) \|_{L^2(\Omega)^d}^2\right),}
\end{align}
the constant between the parentheses being strictly positive owing to  Assumption~(H\ref{H.weak-coupling}). 
Combining the above estimate with~\eqref{eq:piH1.1}, and since $\Ff_\eps(\bX_k) \geq \frac M2 \| \theta_k \|_{L^2(\O)}^2$, one gets that 
\be\label{eq:NRG.chi.k1}
\left\|\grad (G_\eps(\phi_k) + \pi_k)\right\|_{L^2(\O)^d}^2  \geq 
\cter{cte:H1.pi} \eps\left\|\pi_k \right\|_{H^1(\O)}^2 - 4 \eps \Ff_\eps(\bX_k) \clemrev{-\cter{cte:H1.pi}^{-1} \eps \left( \|\grad \phi_r \|_{L^2(\Omega)^d}^2 +b^2  \|\gravmech(\tilde \bphi) \|_{L^2(\Omega)^d}^2 \right)}.
\ee   
for $\ctel{cte:H1.pi} = \left(\frac1M -  \frac{b^2 \cter{cte:mecha}}{\lambda}\right)>0$. 
As a consequence, \eqref{eq:piH1.0} yields 
\be\label{eq:piH1.3}
\int_\O \sum_{\a \in\{n,w\}} s_{\a,k} |\grad p_{\a,k}|^2 \geq 
\cter{cte:H1.pi} \eps \left\|\pi_k \right\|_{H^1(\O)}^2  - 4 \eps \Ff_\eps(\bX_k)  \clemrev{-\cter{cte:H1.pi}^{-1} \eps \left( \|\grad \phi_r \|_{L^2(\Omega)^d}^2 +b^2  \|\gravmech(\tilde \bphi) \|_{L^2(\Omega)^d}^2 \right)},
\ee
and the last term enters an updated definition of $\cter{cte:NRG.1}$ for a uniformly bounded $\eps$. We collect \eqref{eq:Ak}, \eqref{eq:piH1.-1} and \eqref{eq:piH1.3} in \eqref{eq:ADRk} to recover~\eqref{eq:Pregdis.NRG}. 

For $\eps$ small enough so that $ \cter{cte:Pregdis.NRG.2} h \eps^2 \leq 1$, one deduces from~\eqref{eq:Pregdis.NRG} that 
\be\label{eq:NRG.tralala}
{\|\grad (G_\eps(\phi_k) + \pi_k)\|}_{L^2(\O)^d}^2 + \eps {\|\pi_k \|}^2_{H^1(\O)} \leq \frac{C}{\eps} \left(1 + \frac1{{h}} \right) \ee
for some $C$ depending neither  on $k$, $\tilde \bphi$, $\eps$ nor on  $h$ ($C$ will vary along the following lines, but will remain independent on the aforementioned parameters). Using the elementary inequality $\| a + b \|^2 \leq 2 \| a \|^2 + 2 \| b \|^2$, one further gets that 
\be\label{eq:NRG.tralala2}
\|
 \grad G_\eps(\phi_k) 
\|_{L^2(\O)^d}^2 \leq 2 \|
 \grad \left(G_\eps(\phi_k) +\pi_k \right)
\|_{L^2(\O)^d}^2  + 2 \|
\grad \pi_k
\|_{L^2(\O)^d}^2 \leq \frac{C}{\eps} \left(1 + \frac1{{h}} \right) \left(1 + \frac1\eps \right).
\ee
 Moreover, the relation 
\[
G_\eps(\phi_k) = s_{n,k} p_{n,k} + s_{w,k} p_{w,k} - \gamma_\eps(s_{n,k}) - \pi_k
\]
holds everywhere in $\O$, and in particular also on $\G^D$, where 
\[
| G_\eps(\phi_k) | \leq | p_n^D| + |p_w^D| + \gamma_\eps(1) + |\pi_k| \leq | p_n^D| + |p_w^D| + \gamma(1) + \eps + |\pi_k|. 
\]
We infer from a trace theorem that
\[
\| \pi_k \|_{L^2(\G^D)} \leq C \| \pi_k \|_{H^1(\O)}.
\]
Therefore, we get that 
\[
\| G_\eps(\phi_k) \|_{L^2(\G^D)} \leq C(1+\|\pi_k\|_{H^1(\O)}), 
\]
and thus that 
\be\label{eq:chiH1.k}
\| G_\eps(\phi_k) \|_{H^1(\O)} \leq C(1+\|\pi_k\|_{H^1(\O)} +  \| \grad G_\eps(\phi_k) \|_{L^2(\O)^d}) 
\ee
since $y \mapsto \| \grad y \|_{L^2(\O)^d} + \|y\|_{L^2(\G^D)}$ is equivalent to the usual $H^1(\O)$ norm. 
Incorporating~\eqref{eq:NRG.tralala} and \eqref{eq:NRG.tralala2} in \eqref{eq:chiH1.k} gives~\eqref{eq:NRG.chi.k}, concluding the proof of Proposition~\ref{prop:Pregdis.NRG}.
\end{proof}

\subsection{Passing to the limit $k\to+\infty$}\label{ssec:ktoinf}

This section is devoted to the proof of the following proposition, which is deduced from Propositions~\ref{prop:Pregdis} and \ref{prop:Pregdis.NRG} after letting $k$ tend to $+\infty$.

\begin{prop}\label{prop:ktoinf}
Let $\bphi^\star \in L^\infty(\O; \R_+^2)$ with $\bphi^\star \in \Kk_\bphi$ a.e. in $\O$, and let $\bu^\star \in V^d$ and $\pi^\star = M \theta^\star$ in $L^2(\O)$ be such that 
\eqref{eq:constraint.star} holds. Assume that $\cter{cte:Pregdis.NRG.2} h \eps^2 \leq 1$, 
then for any $\tilde \bphi \in L^2(\O)^2$ and $\tilde \bu \in \bU$, there exists a unique solution $(\bphi, \bp, \bu, \theta, \pi)$ with $M\theta = \pi$ and  $\bp - \bp^D \in V^2$ to the following problem:
\begin{subequations}
\be\label{eq:mecha.ktoinf}
\int_\O \bsig: \beps(\bv) = \int_\O b\,\pi \div \bv + \int_\O \gravmech(\tilde \bphi) \cdot \bv, \quad \forall \bv \in V^d, \quad \text{with}\;\; \bsig = 2\mu\, \beps(\bu) + \lambda\, \div \bu\, \bbI, 
\ee
such that $\bphi = (\phi_{n}, \phi_{w}) = \bPhi_\eps(\bp- \pi)$ satisfies
\be\label{eq:constraint.ktoinf}
\phi - b\, \div \bu - \theta = \phi_r.
\ee
and such that
\be\label{eq:ktoinf.cons}
\int_\O \frac{\phi_{\a} - \phi_\a^\star}{h} v + \int_\O \frac{k_\eps(\tilde s_\a)}{\mu_\a} \bbK(\tilde \phi) \grad \left(p_{\a} - \rho_\a \g\cdot (\x + \tilde \bu)\right) \cdot \grad v = 0 \quad \text{for all}\; v \in V, \; \a \in \{n,w\}.
\ee
\end{subequations}
Moreover, it satisfies  
\begin{multline}
\left( 1 - \cter{cte:Pregdis.NRG.2} h \eps^2 \right) \Ff_\eps(\bX) \\
+ {h}\eps \cter{cte:Pregdis.NRG}  \left( \sum_{\a \in\{n,w\}}  \left\|\grad p_{\a} \right\|_{L^2(\O)^d}^2 + \left\|\grad (G_\eps(\phi) + \pi) \right\|_{L^2(\O)^d}^2+ \eps \left\| \pi \right\|_{H^1(\O)}^2\right) \\
\leq \Ff_\eps(\bX^\star) + \int_\O \gravmech(\tilde \bphi) \cdot (\bu - \bu^\star)  \\
+ \sum_{\a \in \{n,w\}} \int_\O (\phi_{\a} - \phi_\a^\star) p_\a^D + \cter{cte:NRG.1} {h}\left(1 + \|\tilde \bu \|_{V^d}^2 +  \|\tilde \bphi \|_{L^2(\O)}^2\right)
\label{eq:ktoinf.NRG}
\end{multline}
and 
\be\label{eq:NRG.chi}
\| \pi \|_{H^1(\O)} + \|G_\eps(\phi) \|_{H^1(\O)} \leq  \cter{cte:chi}. 
\ee
\end{prop}
\begin{proof}
The a priori estimates derived in Proposition~\ref{prop:Pregdis.NRG} can be refined. Using Korn's inequality and 
\be\label{eq:aux.ktoinf.NRG}
\int_\O \gravmech(\tilde \bphi) \cdot (\bu - \bu^\star) \leq \frac12 \left( 1 - \cter{cte:Pregdis.NRG.2} h \eps^2 \right) \Ff_\eps(\bX_k) + C \left( \|\bu^\star\|_{L^2(\O)^d}^2 + \| \gravmech(\tilde \bphi) \|_{L^2(\Omega)}^2\right)
\ee
for some constant $C>0$ depending Korn's constant, Lam\'e parameters, as well as $\left( 1 - \cter{cte:Pregdis.NRG.2} h \eps^2 \right)^{-1}$, thus being uniformly bounded. Starting from~\eqref{eq:Pregdis.NRG}, the state-dependent contribution can be compensated in the Helmholtz energy $\Ff_\eps(\bX_k)$ on the left hand side of~\eqref{eq:Pregdis.NRG}. The remaining terms on the right hand side of~\eqref{eq:aux.ktoinf.NRG} are uniformly bounded in $k$, given $\tilde \bphi \in L^2(\O)^2$ and $\tilde \bu \in \bU$. Together with $\phi_{\a,k}\in \mathcal{K}_\phi$ and this results in a uniform bound in $k$.
This ensures the existence of some $\bp \in \bp^D + V^2$, and $\pi, \theta \in H^1(\O)$,  such that, up to a subsequence, 
\begin{subequations}
\begin{align}
\label{eq:pktop}
\bp_k &\underset{k\to+\infty} \longrightarrow \bp \quad \text{a.e. in $\O$ and weakly in $H^1(\O)^2$}, \\
\label{eq:piktopi}
\pi_k &\underset{k\to+\infty} \longrightarrow \pi \quad \text{a.e. in $\O$ and weakly in $H^1(\O)$}, \\
\nonumber
\theta_k &\underset{k\to+\infty} \longrightarrow \theta \quad \text{a.e. in $\O$ and weakly in $H^1(\O)$},
\end{align}
\end{subequations}
with $M\theta = \pi$. The control of the energy $\Ff_\eps(\bX)$ provides a uniform control on the $H^1(\O)^d$ norm of $\bu_k$ thanks to Korn's inequality, and even a uniform control on the $H^1(\O)^d$ norm of $\div \bu_k$ thanks to Assumption~(H\ref{H.Omega}). Therefore, there exists $\bu \in \bU$ such that 
\[
\bu_k \underset{k\to+\infty} \longrightarrow \bu \quad \text{weakly in $H^1(\O)$} \quad  \text{and} 
\quad \div \bu_k \underset{k\to+\infty} \longrightarrow \div \bu \quad \text{weakly in $H^1(\O)$}.
\]
Since $\bPhi_\eps$ is continuous owing to Lemma~\ref{lem:Phi}, one infers from~\eqref{eq:pktop} and \eqref{eq:piktopi} that 
\[
\bphi_k = \bPhi_\eps(\bp_k - \pi_k) \underset{k\to+\infty} \longrightarrow \bPhi_\eps(\bp - \pi) = \bphi \quad \text{a.e. in $\O$}, 
\]
with $\bphi \in \Kk_\bphi$ a.e. in $\O$ since $\bphi_k$ does. The aforementioned convergences are enough to pass to the limit in \eqref{eq:mecha.k}, which gives \eqref{eq:mecha.ktoinf}, and in~\eqref{eq:Pregdis.cons.k}, leading to~\eqref{eq:ktoinf.cons}. 
Thanks to~\eqref{eq:Wk.dense}, passing to the limit in~\eqref{eq:constraint.k} provides that 
\[
\int_\O (\phi - b \div \bu - \theta - \phi_r) w = 0, \qquad \forall w \in L^2(\O), 
\]
which is equivalent to claiming to \eqref{eq:constraint.ktoinf} holds in $L^2(\O)$ and thus almost everywhere and in $H^1(\O)$.
Inequality~\eqref{eq:ktoinf.NRG} is recovered from~\eqref{eq:Pregdis.NRG} by invoking the weak lower semi-continuity of the left-hand side and linearity of the right hand side. 
\end{proof}

\subsection{Unfreezing the mobility and recovering nonlinear gravitational energy}\label{ssec:Schauder}
In this section, we establish the existence of a solution to the previous problem with the additional constraints that $\bphi = \tilde \bphi$ and $\tilde{\bu} = \bu$.

\begin{prop}\label{prop:Schauder}
Let $\bphi^\star \in L^\infty(\O)^2$ with $\bphi^\star \in \Kk_\bphi$ a.e. in $\O$, and let $\bu^\star \in \bU$ and $\pi^\star = M \theta^\star$ in $L^2(\O)$ be such that~\eqref{eq:constraint.star} holds. Let $h \in \left(0, \frac{1}{\cter{cte:radius}}\right)$, with $\cter{cte:radius}$ defined below.
Then
 there exists a solution $(\bphi, \bp, \bu, \theta, \pi)$ with $\bp - \bp^D \in V^2$, $\bphi \in H^1(\O)^2$, $\bu \in \bU$, $\pi \in H^1(\O)$, $G_\eps(\phi) \in H^1(\O)$ and $M \theta = \pi$ to the following problem:
\be\label{eq:mecha.schauder}
\int_\O \bsig: \beps(\bv) = \int_\O b\,   \pi \div \bv + \int_\O \gravmech(\bphi) \cdot \bv, \quad \forall \bv \in V^d, \quad \text{with}\;\; \bsig = 2\mu\, \beps(\bu) + \lambda\, \div \bu\, \bbI, 
\ee
such that $\bphi = (\phi_{n}, \phi_{w}) = \bPhi_\eps(\bp- \pi)$ satisfies
\be\label{eq:constraint.schauder}
\phi - b\, \div \bu - \theta = \phi_r.
\ee
and such that
\be\label{eq:schauder.cons}
\int_\O \frac{\phi_{\a} - \phi_\a^\star}{h} v + \int_\O \frac{k_\eps(s_\a)}{\mu_\a} \bbK( \phi) \grad \left( p_{\a} - \rho_\a \g\cdot (\bx + \bu)\right) \cdot \grad v = 0 \quad \text{for all}\; v \in V, \; \a \in \{n,w\}.
\ee
\end{prop}

\begin{proof}
Let $\Tt: L^2(\O)^2 \times V^d \to L^2(\O)^2 \times V^d$ mapping $(\tilde \bphi, \tilde \bu)$ to $(\bphi, \bu)$ as in Proposition~\ref{prop:ktoinf}. Let the product space be equipped by the product norm $\triplenorm{\cdot}$ 
defined through
\[
\triplenorm{(\bphi, \bu )}^2:= \| \bphi \|_{L^2(\O)^2}^2 + \| \bu \|_{\mu,\lambda}^2 
\quad (\bphi,\bu)\in L^2(\O)^2 \times V^d,
\]
where we have set 
\[
 \| \bu \|_{\mu,\lambda}^2 = \int_\O 2\mu \, \beps(\bu) : \beps(\bu) + \lambda (\div \bu)^2, \qquad \forall \bu \in V^d.
\]
The norm $\|\cdot\|_{\mu,\lambda}$ is equivalent to the $H^1(\O)^d$ norm on $V^d$ thanks to Poincaré's and Korn's inequalities.

First, since $\bphi$ takes its values in the bounded set $\Kk_\bphi$ of $\R^2$, and since $\O$ is bounded, then 
$\|\bphi\|_{L^2(\O)^2}^2 \leq R_\phi$ for some $R_\phi$ not depending on $\tilde \bphi$; here, we recall the use of the orthogonal projection $\bPi$ onto $\Kk_\bphi$ to freeze the porosities, cf.~\eqref{eq:tilde.sphi}. Furthermore, from~\eqref{eq:ktoinf.NRG} we can infer that
\[
 \| \bu \|_{\mu,\lambda}^2
 \leq C\left(1 + h  \|\tilde \bu \|_{V^d}^2 + h \| \tilde \bphi \|_{L^2(\O)^2}^2\right)
 \leq \cter{cte:radius}\left(1 + h \| \tilde \bu \|_{\mu, \lambda}^2 + h R_\phi \right). 
\]
for suitable constant $C,\ctel{cte:radius}>0$ independent of $\bu$ and $\tilde{\bu}$, where the latter bound follows from Korn's inequality. Thus, assuming $\cter{cte:radius}h < 1$ and $\|\tilde \bu \|_{\mu, \lambda}^2 \leq R_u := \tfrac{\cter{cte:radius}(1+h R_\phi)}{1-\cter{cte:radius}h}$, then
\[
\triplenorm{( \bphi, \bu )}^2 \leq R_\phi + R_u.
\]
In particular, $\Tt$ maps the ball of radius $(R_\phi + R_u)^{1/2}$ of $L^2(\O)^2 \times V^d$ into itself.

Second, let us show that $\Tt$ is compact. For this, let $(\tilde \bphi_k, \tilde \bu_k)_{k\geq 1} \subset L^2(\O)^2 \times V^d$ be a bounded sequence in $L^2(\O)^2 \times H^1(\Omega)^d$. Then, up to subsequence, there exist $(\tilde \bphi^\star, \tilde \bu^\star) \subset L^2(\O)^2 \times V^d$ such that
\be\label{eq:schauder_conv.1}
 \tilde \bphi_k \underset{k \to \infty}\longrightarrow \tilde \bphi^\star \quad \text{weakly in }L^2(\O)^2,\qquad 
 \tilde \bu_k \underset{k \to \infty}\longrightarrow \tilde \bu^\star \quad \text{weakly in }H^1(\Omega)^d.
\ee
The latter implies
\be\label{eq:schauder_conv.2}
 \tilde \bu_k \underset{k \to \infty}\longrightarrow \tilde \bu^\star \quad \text{strongly in }L^2(\O)^d.
\ee
Consider the corresponding sequence $(\bphi_k,\bu_k)_{k\geq1}$ defined by $(\bphi_k,\bu_k) = \mathcal{T}(\tilde \bphi_k, \tilde\bu_k)$, accompanied with $\bX_k$ solving~\eqref{eq:mecha.ktoinf}--\eqref{eq:ktoinf.cons} and the uniform stability bound~\eqref{eq:ktoinf.NRG} with constants independent of $k$ and with $\Ff(\bX_k) \geq 0$. From~\eqref{eq:ktoinf.NRG}, we infer that, up to a subsequence,
\begin{subequations}
 \label{eq:schauder_conv.3}
\begin{alignat}{3}
 \bu_k &\underset{k \to \infty}\longrightarrow \bu^\star &\quad& \text{weakly in }H^1(\Omega)^d\text{ and strongly in }L^2(\O)^d\\
 \pi_k &\underset{k \to \infty}\longrightarrow \pi^\star &\quad& \text{weakly in }H^1(\O)\text{ and strongly in }L^2(\O)
\end{alignat}
\end{subequations}
for some $\bu^\star \in V^d$ and $\pi^\star \in H^1(\O)$. 
Due the linearity of~\eqref{eq:mecha.ktoinf}, cf.~(\ref{H.gravity}), we can consider the limit $k\rightarrow \infty$, yielding
\be\label{eq:schauder_el.1}
\int_\O 2\mu\, \beps(\bu^\star): \beps(\bv) + \lambda (\div \bu^\star)(\div \bv)= \int_\O b\,\pi^\star \div \bv + \int_\O \gravmech(\tilde \bphi^\star) \cdot \bv, \quad \forall \bv \in V^d.
\ee
In addition, testing~\eqref{eq:mecha.ktoinf} with $\bv = \bu_k$ yields
\be\label{eq:schauder_el.2}
\int_\O 2\mu \, \beps(\bu_k) : \beps(\bu_k) + \lambda (\div \bu_k)^2 = \int_\O b\,\pi_k \div \bu_k + \int_\O \gravmech(\tilde \bphi_k) \cdot \bu_k.
\ee
From the convergences~\eqref{eq:schauder_conv.1} and~\eqref{eq:schauder_conv.3}, combined in suitable pairs of strong and weak convergence, we can infer for the right hand side of \eqref{eq:schauder_el.2} that
\[
\int_\O b\,\pi_k \div \bu_k + \int_\O \gravmech(\tilde \bphi_k) \cdot \bu_k
\underset{k \to \infty}\longrightarrow
\int_\O b\,\pi^\star \div \bu^\star + \int_\O \gravmech(\tilde \bphi^\star) \cdot \bu^\star,
\]
which implies convergence of the left hand side by combining~\eqref{eq:schauder_el.1} (tested with $\bv = \bu^\star$) and~\eqref{eq:schauder_el.2}
\[
 \|\bu_k\|_{\mu,\lambda}^2 = \int_\O 2\mu \, \beps(\bu_k) : \beps(\bu_k) + \lambda (\div \bu_k)^2
 \underset{k \to \infty}\longrightarrow
 \int_\O 2\mu \, \beps(\bu^\star) : \beps(\bu^\star) + \lambda (\div \bu^\star)^2 = \|\bu^\star\|_{\mu,\lambda}^2.
\]
Norm convergence in $\|\cdot\|_{\mu,\lambda}$ together with weak convergence, implies strong convergence
\be\label{eq:schauer_conv.1}
 \bu_k \underset{k \to \infty}\longrightarrow \bu^\star \quad \text{strongly in }H^1(\Omega)^d.
\ee
Furthermore, we deduce from the fact that $\bphi_k = \bPhi_\eps(\bp_k -\pi_k)$ and from the Lipschitz continuity of $\bPhi_\eps$, cf. Lemma~\ref{lem:Phi}, that 
\[
\|\grad \bphi_k\|_{L^2(\O)^{d\times 2}} \leq \Lambda_\eps \|\grad \bp_k - \grad \pi_k\|_{L^2(\O)^{d\times 2}} \leq \Lambda_\eps \left(\|\grad \bp_k \|_{L^2(\O)^{d\times 2}} + \|\grad \pi_k\|_{L^2(\O)^{d}}\right)
\]
for some $\Lambda_\eps$ depending on $\eps$ (and the uniform bound on $(\tilde \bphi_k, \tilde \bu_k)_{k\geq 1}$). Then we infer from~\eqref{eq:ktoinf.NRG} that $\|\grad \bphi_k \|_{L^2(\O)^{d\times 2}} \leq C$ for some $C$, and we conclude
\be\label{eq:schauer_conv.2}
 \bphi_k \underset{k \to \infty}\longrightarrow \bphi^\star \quad \text{weakly in }H^1(\Omega)^d\text{ and strongly in }L^2(\O).
\ee 
Thus, finally, from~\eqref{eq:schauer_conv.1} and~\eqref{eq:schauer_conv.2}, it follows that $\Tt$ is compact.

Continuity of $\Tt$ follows along the proof of compactness. Indeed, the derived convergences are sufficient to infer respective convergence of the accompanying sequences $(\theta_k)_{k\geq 1}$ and $(s_{\a,k})_{k\geq 1}$.
Therefore, we can pass to the limit in~\eqref{eq:mecha.ktoinf}, \eqref{eq:constraint.ktoinf} and \eqref{eq:ktoinf.cons}, showing that $(\bphi^\star, \bu^\star)$ is the solution to the problem described in Proposition~\ref{prop:ktoinf} corresponding $(\tilde\bphi^\star, \tilde \bu^\star)$. From the uniqueness result stated in Proposition~\ref{prop:ktoinf}, we deduce that $(\bphi^\star, \bu^\star) = \Tt(\tilde \bphi^\star, \tilde \bu^\star)$, and the continuity of $\Tt$ follows.

The operator $\Tt$ then fulfills all the assumptions of the Schauder fixed point theorem, ensuring the existence of (at least) one fixed point for $\Tt$. This concludes the proof of Proposition~\ref{prop:Schauder}. 
\end{proof}

\section{The semi-discrete in time system without regularization}\label{sec:noreg}

Our goal in this section is to get rid of the two regularizations we incorporated in the system in Section~\ref{ssec:reg}, that are: 
\begin{enumerate}[(i)]
\item the regularization~\eqref{eq:kdelta} of the mobilities to make them non-degenerate; 
\item the regularization of the capillary energy density function $F_\eps$ introduced in Lemma~\ref{lem:Phi.2} to make it uniformly convex. 
\end{enumerate}
To this end, we first derive uniform estimates w.r.t. $\eps$, to be used to let it tend to $0$.

\subsection{Uniform estimates w.r.t. $\eps$ and $h$}\label{ssec:noreg.NRG}

Let us start from a solution $(\bphi_\eps, \bp_\eps, \bu_\eps, \theta_\eps, \pi_\eps)$ to the regularized semi-discrete system as constructed in Proposition~\ref{prop:Schauder}. We stress here the dependence of the solution in the regularization parameter $\eps>0$.

The core result of this section is the following energy estimate, which involves the increasing continuous function $\xi_\eps:[0,1] \to \R_+$ defined by 
\[
\xi_\eps(s) = \int_0^s \sqrt{a(1-a)} \gamma_\eps''(a) \d a, \qquad \forall s \in [0,1]. 
\]
\begin{prop}\label{prop:noreg.NRG}
Let $\eps \in (0,\frac14]$ and $h \in \left(0, \frac{1}{\cter{cte:radius}}\right)$, then the solution $(\bphi_\eps, \bp_\eps, \bu_\eps, \theta_\eps, \pi_\eps)$ to the semi-discrete in time and regularized problem fulfills the uniform w.r.t. $\eps$ and $h$ estimate 
\begin{multline}\label{eq:noreg.NRG}
(1 - \cter{cte:piH1.noreg.1} h) \Ff_\eps(\bX_\eps)  + \frac{K_\flat}{4 \mu^\sharp} h \bigg( 2 \| \grad \xi_\eps(s_{n,\eps}) \|^2_{L^2(\O)^d} + \cter{cte:piH1.noreg} \| \pi_\eps\|^2_{H^1(\O)} \\+   \| \grad(G_\eps(\phi_\eps) + \pi_\eps)\|^2_{L^2(\O)^d} + {\eps} \sum_{\a\in\{n,w\}} \| \grad \hat p_{\a,\eps}\|^2_{L^2(\O)^d} \bigg) \\
\leq \Ff_\eps(\bX^\star) + \int_\O \gravmech(\bphi_\eps) \cdot (\bu_\eps - \bu^\star) + \sum_{\a \in \{n,w\}} \int_\O (\phi_{\a,\eps} - \phi_\a^\star)\, p_\a^D +  \cter{cte:NRG.1} {h}, 
\end{multline}
with $\ctel{cte:piH1.noreg.1}>0$ and $\ctel{cte:piH1.noreg}>0$ depending neither on $\eps$ nor on $\bX^\star$ and nor on $h$. 
\end{prop}

\begin{proof} The proof shares several features with the one of Proposition~\ref{prop:Pregdis.NRG}. Choosing $v = h(p_{\a,\eps} - p_\a^D)$ in~\eqref{eq:schauder.cons} 
and summing over $\a \in \{n,w\}$, then proceeding similarly as in the proof of Proposition~\ref{prop:Pregdis.NRG} gives, cf.\ in particular~\eqref{eq:ADRk}, 
\be\label{eq:ADReps}
\Aa_\eps + \Bb_\eps +  \frac12 \Dd_\eps \leq \Rr_{\eps} + \cter{cte:NRG.1} {h} (1 + \|\bu_\eps \|_{V^d}^2),
\ee
 where we have set 
 \[
\Aa_\eps = \sum_{\a \in \{n,w\}} \int_\O (\phi_{\a,\eps} - \phi_\a^\star) (p_{\a,\eps} - \pi_\eps), \qquad
\Bb_\eps = \int_\O (\phi_{\eps} - \phi^\star) \pi_\eps, 
\]
\[
\Dd_\eps = h\sum_{\a \in \{n,w\}}  \int_\O \frac{k_\eps(s_{\a,\eps})}{\mu_\a} \bbK(\phi_\eps) \grad p_{\a,\eps} \cdot \grad p_{\a,\eps} \quad\text{and}\quad
\Rr_\eps =  \sum_{\a \in \{n,w\}} \int_\O (\phi_{\a,\eps} - \phi_\a^\star)\, p_\a^D.
\]
As in the proof of Propostion~\ref{prop:Pregdis.NRG}, cf.\ also~\eqref{eq:Ak}, again utilizing the relation $\bphi_\eps = \bPhi_\eps(\bp_\eps - \pi_\eps)$ and Lemma~\ref{lem:Phi.2}, we end up with
\be\label{eq:noreg.NRG.AB}
\Aa_\eps + \Bb_\eps \geq \Ff_\eps(\bX_\eps) - \Ff_\eps(\bX^\star) - \int_\O \gravmech(\bphi_\eps)\cdot (\bu_\eps - \bu^\star).
\ee

Concerning the term $\Dd_\eps$, we start by noticing that $k_\eps(s_{\a,\eps}) \geq (s_{\a,\eps}+\eps)/2$. Then taking inspiration on what was done in Section~\ref{ssec:2.weak}, we underestimate $\Dd_\eps$ by 
\[
\Dd_\eps \geq \Dd_{1,\eps} + \Dd_{2,\eps}
\]
with 
\[
\Dd_{1,\eps}  =  \frac{K_\flat}{2 \mu^\sharp}h \eps \int_\O \sum_{\a\in\{n,w\}}  |\grad p_{\a,\eps}|^2 \quad \text{and} \quad 
\Dd_{2,\eps} =  \frac{K_\flat}{2 \mu^\sharp}h  \int_\O \sum_{\a\in\{n,w\}} s_{\a,\eps} |\grad p_{\a,\eps}|^2.
\]
Using the elementary inequality $(a+b)^2 \geq a^2/2 - b^2$ and the decomposition 
$p_{\a,\eps} = \hat p_{\a,\eps} + \pi_\eps + G_\eps(\phi_\eps)$ of the phase pressures, we get that 
\be\label{eq:noreg.D1eps}
\Dd_{1,\eps} \geq  \frac{K_\flat}{4 \mu^\sharp}h \eps \int_\O \sum_{\a\in\{n,w\}}  |\grad \hat p_{\a,\eps}|^2 
- \frac{K_\flat}{\mu^\sharp}h \eps \int_\O |\grad (G_\eps(\phi_\eps) + \pi_\eps)|^2.
\ee
Besides,
from similar arguments to those presented in Section~\ref{ssec:2.weak}, we can rewrite 
\be\label{eq:noreg.NRG.2}
 \int_\O \sum_{\a\in\{n,w\}} s_{\a,\eps} |\grad p_{\a,\eps}|^2 = \int_\O s_{n,\eps} s_{w,\eps} |\grad (p_{n,\eps} - p_{w,\eps})|^2 + \int_\O |\grad (G_\eps(\phi_\eps) + \pi_\eps)|^2. 
\ee
So for $\eps \leq 1/4$, we get that 
\begin{align}\label{eq:noreg.Deps}
\Dd_\eps \geq &\; \frac{K_\flat}{2 \mu^\sharp}h  \int_\O s_{n,\eps} s_{w,\eps} |\grad (p_{n,\eps} - p_{w,\eps})|^2 \\
& \nonumber \;+ 
 \frac{K_\flat}{4 \mu^\sharp}h \eps \int_\O \sum_{\a\in\{n,w\}}  |\grad \hat p_{\a,\eps}|^2 +  \frac{K_\flat}{4 \mu^\sharp}h\int_\O |\grad (G_\eps(\phi_\eps) + \pi_\eps)|^2. 
\end{align}
For the first term in the above right-hand side, either $p_{n,\eps} - p_{w,\eps} = \gamma'(s_{n,\eps})$ or $s_{n,\eps} s_{w,\eps} = 0$. Therefore, 
\be\label{eq:noreg.NRG.5}
 \int_\O s_{n,\eps} s_{w,\eps} |\grad (p_{n,\eps} - p_{w,\eps})|^2 = \int_\O s_{n,\eps} (1-s_{n,\eps}) |\grad \gamma_\eps'(s_{n,\eps})|^2 = \int_\O |\grad \xi_\eps(s_{n,\eps})|^2. 
 \ee
Concerning the last term in the right-hand side of~\eqref{eq:noreg.Deps}, the approach adopted in the proof of Proposition~\ref{prop:Pregdis.NRG} leads to a control depending on $\eps$, so modifications are needed. 

We can use $\pi_\eps = M\theta_\eps$,~\eqref{eq:constraint.schauder}, and the fact that 
$\phi_r\in H^1(\Omega)$ to rewrite 
\begin{align*}
 |\grad (G_\eps(\phi_\eps) + \pi_\eps)|^2 &=  |\grad \pi_\eps|^2 + |\grad G_\eps(\phi_\eps)|^2  + 2 \grad G_\eps(\phi_\eps) \cdot \grad \pi_\eps \\
 &= 
 |\grad \pi_\eps|^2 + |\grad G_\eps(\phi_\eps)|^2 + 2 M \grad G_\eps(\phi_\eps) \cdot \grad \left(\phi_\eps - b(\div \bu_\eps) - \phi_r\right).
\end{align*}
So Young's inequalityfor any $\xi_1>0$ and the monotonicty of $G_\eps$ provide that
\begin{align}\label{eq:noreg.NRG.3}
|\grad (G_\eps(\phi_\eps) + \pi_\eps)|^2 &\geq |\grad \pi_\eps|^2  - M^2 |\grad (b\div \bu_\eps + \phi_r)|^2 \\
&\geq |\grad \pi_\eps|^2  - (bM)^2 \left(1 + \xi_1 \right) |\grad (\div \bu_\eps) |^2 -  M^2\left(1 + \frac1{4\xi_1}\right)|\grad \phi_r |^2,
\nn
\end{align}
On the other hand, we infer from Assumption~(H\ref{H.Omega}), recalling~\eqref{eq:D2uk}, and applying Young's inequality for any $\xi_2>0$, that 
\[
\int_\O |\grad (\div \bu_\eps)|^2 \leq \left( \frac{b^2\,  \cter{cte:mecha}}{\lambda}\right)^2\left( 1 + \xi_2\right) \| \pi_\eps\|^2_{H^1(\O)} + 
\left(\frac{\cter{cte:mecha}}\lambda\right)^2 \left(1 + \frac1{4\xi_2} \right) \| \gravmech(\bphi_\eps) \|_{L^2(\O)^d}^2.
\]
Then we deduce from~\eqref{eq:noreg.NRG.3} that 
\begin{align*}
\|\grad (G_\eps(\phi_\eps) + \pi_\eps)\|_{L^2(\Omega)^d}^2 
\geq
\left(1 - \left(\frac{\cter{cte:mecha} Mb^2}{\lambda} \right)^2 \left( 1 + \xi_1\right)\left( 1 + \xi_2\right) \right)\|\pi_\eps\|_{H^1(\O)}^2 - M^2 \| \theta_\eps \|_{L^2(\O)}^2\\
- M^2\left(1 + \frac1{4\xi_1}\right)|\grad \phi_r |_{L^2(\O)}^2 - \left(1 + \frac1{4\xi_2}\right) \| \gravmech(\bphi_\eps) \|_{L^2(\O)^d}^2 .
\end{align*}
Under the weak coupling assumption (H\ref{H.weak-coupling}), one can choose $\xi_1>0$ and $\xi_2 \in (0,1)$ to satisfy
\[
1+\xi_1 = \frac{1}{1 - \xi_2},\quad\text{and} \quad
 \xi_2 = \frac{\lambda^2 - (M\cter{cte:mecha}b^2)^2}{\lambda^2 + 3 (M\cter{cte:mecha}b^2)^2} \in (0,1),
\]
and together with the uniform bounds on $\gravmech$ and $\phi_r$ following from (H\ref{H.Cte}) and (H\ref{H.gravity}),
one obtains
 \be\label{eq:noreg.NRG.4}
\int_\O  |\grad (G_\eps(\phi_\eps) + \pi_\eps)|^2 \geq\cter{cte:piH1.noreg} \|\pi_\eps\|^2_{H^1(\O)} - \cter{cte:piH1.noreg.1} \frac{M}2 \| \theta_\eps \|_{L^2(\O)}^2
{- \cter{cte:noreg.NRG.4}}.
\ee
with $\cter{cte:piH1.noreg}>0$, $\cter{cte:piH1.noreg.1}>0$, and $\ctel{cte:noreg.NRG.4}$ independent on $\eps$, $\bX^\star$ and $h$.
Combining \eqref{eq:noreg.NRG.AB}--\eqref{eq:noreg.NRG.5} and \eqref{eq:noreg.NRG.4} in~\eqref{eq:ADReps} and noticing again that $ M \| \theta_\eps \|_{L^2(\O)}^2 \leq 2 \Ff_\eps (\bX_\eps)$, in addition to $\|\bu_\eps\|_{V^d}^2 \leq C\Ff_\eps(\bX_\eps)$ due to Korn's inequality (for some $C>0$ for simplicity entering the definition of $\cter{cte:piH1.noreg.1}$), to be used in~\eqref{eq:ADReps}, provides the desired  estimate~\eqref{eq:noreg.NRG}.
\end{proof}

\begin{lem}\label{lem:noreg.NRG}
There exists $\ctel{cte:noreg.42}>0$ such that, for all $\eps \in (0,\frac14]$ and all $h \in (0,\frac1{\cter{cte:piH1.noreg.1}}]$ (without loss of generality, we assume from now on $\cter{cte:radius} \leq \cter{cte:piH1.noreg.1}$), there holds 
\[
{\| G_\eps(\phi_\eps) \|}_{H^1(\O)}^2 \leq \cter{cte:noreg.42} \left( 1 + \frac1h \right).
\]
\end{lem}
\begin{proof}
We will adapt here the program to derive~\eqref{eq:NRG.chi.k}. As $\cter{cte:piH1.noreg.1} h \leq 1$,  and since $\Ff_\eps(\bX_\eps) \geq 0$ while $\Ff_\eps(\bX^\star)$ is finite, we infer from~\eqref{eq:noreg.NRG} that 
\[
\| \pi_\eps\|^2_{H^1(\O)} +   \| \grad(G_\eps(\phi_\eps) + \pi_\eps)\|^2_{L^2(\O)^d} \leq C \left( 1 + \frac1h \right),
\]
where $C$ denotes again a generic quantity independent of $\eps$ and $h$. We note that the state-dependent terms are either uniformly bounded, or (like the $\bu_\eps$-dependent term) can be compensated on the left hand side under (\ref{H.gravity}) also affecting the value of $C$. Then updating~\eqref{eq:NRG.tralala2} leads to 
\[
\|
 \grad G_\eps(\phi_\eps) 
\|_{L^2(\O)^d}^2 \leq 2 \|
 \grad \left(G_\eps(\phi_\eps) +\pi_\eps \right)
\|_{L^2(\O)^d}^2  + 2 \|
\grad \pi_\eps
\|_{L^2(\O)^d}^2 \leq C \left(1 + \frac1{{h}} \right).
\]
Then reproducing the arguments detailed for obtaining~\eqref{eq:chiH1.k}, one gets that 
\[
\| G_\eps(\phi_\eps) \|_{H^1(\O)}^2 \leq C(1+\|\pi_\eps\|_{H^1(\O)}^2 +  \| \grad G_\eps(\phi_\eps) \|_{L^2(\O)^d}^2) \leq C\left(1+\frac1h \right), 
\]
concluding the proof of Lemma~\ref{lem:noreg.NRG}.
\end{proof}
\begin{coro}\label{coro:noreg.NRG}
There exist $\ctel{cte:NRG.nreg.1}, \ctel{cte:NRG.nreg.2}>0$ such that, for all $h \in (0,\frac1{2 \cter{cte:piH1.noreg.1}}]$ and $\eps \in (0,\frac14]$, there holds
\begin{multline*}
\Ff_\eps(\bX_\eps) 
+ \cter{cte:NRG.nreg.1} h \left( \|\xi_\eps(s_{n,\eps}) \|^2_{H^1(\O)} + \| \pi_\eps\|^2_{H^1(\O)} + \| G_\eps(\phi_\eps)\|^2_{H^1(\O)}\right)\\
 \leq  (1+2 \cter{cte:piH1.noreg.1}h) \left[\Ff_\eps(\bX^\star)  + \int_\O \gravmech(\bphi_\eps) \cdot (\bu_\eps - \bu^\star) \right] + \sum_{\a \in \{n,w\}} \int_\O (\phi_{\a, \eps} - \phi_\a^\star)\, p_\a^D+ \cter{cte:NRG.nreg.2} {h}.
\end{multline*}
\end{coro}
\begin{proof}
Dividing~\eqref{eq:noreg.NRG} by $1- \cter{cte:piH1.noreg.1}h$, then the elementary inequalities $2 \geq 1+2  \cter{cte:piH1.noreg.1}h \geq (1- \cter{cte:piH1.noreg.1} h)^{-1} \geq 1$ (recall here that we assumed $2 \cter{cte:piH1.noreg.1}h \leq 1$) provide
\begin{multline*}
 \Ff_\eps(\bX_\eps) + C h \left( \|\pi_\eps\|^2_{H^1(\O)} + \|\grad \xi_\eps(s_{n,\eps})\|_{L^2(\O)^d}^2\right) \\
 \leq (1+2  \cter{cte:piH1.noreg.1}h)\left[\Ff_\eps(\bX^\star) 
 +  \int_\O \gravmech(\bphi_\eps) \cdot (\bu_\eps - \bu^\star)
 +\sum_{\a \in \{n,w\}} \int_\O (\phi_{\a,\eps} - \phi_\a^\star)\, p_\a^D \right]
 + \cter{cte:NRG.1} {h}. 
\end{multline*}
Since $\bphi_\eps, \bphi^\star \in \Kk_\bphi$ and since $\bp^D \in L^\infty(\O)^2$ too, 
\[
 \sum_{\a \in \{n,w\}} \int_\O (\phi_{\a,\eps} - \phi_\a^\star)\, p_\a^D \leq C.
\]
Moreover, it follows from (H\ref{H.gamma}) that the function $\xi_\eps$ is uniformly bounded on $[0,1]$, hence we get 
\begin{multline}\label{eq:noreg.NRG++}
\Ff_\eps(\bX_\eps) + C h \left( \|\pi_\eps\|^2_{H^1(\O)} + \|\xi_\eps(s_{n,\eps})\|_{H^1(\O)}^2\right) \\
 \leq (1+2  \cter{cte:piH1.noreg.1}h) \left[\Ff_\eps(\bX^\star) + \int_\O \gravmech(\bphi_\eps) \cdot (\bu_\eps - \bu^\star) \right]
 \sum_{\a \in \{n,w\}} \int_\O (\phi_{\a,\eps} - \phi_\a^\star)\, p_\a^D
 + C' {h}
 \end{multline}
 for some $C'$ also uniform in $\eps$ and $h$. 
 Combining~\eqref{eq:noreg.NRG++} with Lemma~\ref{lem:noreg.NRG} provides the desired result.
\end{proof}

\subsection{Passing to the limit $\eps \to 0$}
Let us get rid of the regularization in the mobilities and in the capillary energy density by passing to the limit $\eps \to 0$ while keeping $h>0$ fixed. The main result of the current section is the following.

\begin{prop}\label{prop:noreg.lim}
Let $h \in (0,1]$, then there exists $\bX = (\bphi, \bu, \theta, \pi)$ with $\bphi \in \Kk_\bphi$ a.e. in $\O$, $\phi \in H^1(\O)$ and $\xi(s_n) - \xi(s_n^D) \in V$, with $\bu \in \bU$, and $\theta, \pi \in H^1(\O)$, as well as some $\chi \in H^1(\O)$ with $\chi \in \ov \chi(\phi)$ such that 
\be\label{eq:noreg.lim.0}
\pi = M\theta \quad \text{and}\quad \phi - b \, \div \bu - \theta = \phi_r, 
\ee
such that 
\be\label{eq:noreg.lim.cons.n}
\int_\O \frac{\phi_n - \phi_n^\star}h v - \int_\O \frac{1}{\mu_n} \bbK(\phi) \left( \grad \psi(s_n) + s_n \grad \left((\pi +\chi) -\rho_n \g \cdot (\bx + \bu) \right)\right) \cdot \grad v = 0, \qquad \forall v \in V, 
\ee
\be\label{eq:noreg.lim.cons.w}
\int_\O \frac{\phi_w - \phi_w^\star}h v - \int_\O \frac{1}{\mu_w} \bbK(\phi) \left( - \grad \psi(s_n) + s_w \grad \left((\pi+\chi) -\rho_w \g \cdot (\bx + \bu) \right) \right) \cdot \grad v = 0, \qquad \forall v \in V, 
\ee
 and
\be\label{eq:noreg.lim.mecha} 
\int_\O \left( 2 \mu \, \beps(\bu) : \beps(\bv) + \lambda (\div \bu)(\div \bv)\right) = \int_\O b \pi \, \div\bv +\int_\O \gravmech(\bphi) \cdot \bv, \qquad \forall \bv \in V^d.
\ee
Moreover, $\bX$ and $\chi$ satisfy
\begin{multline}\label{eq:noreg.lim.NRG}
\Ff(\bX) 
+ \cter{cte:NRG.nreg.1} h \left( \|\xi(s_{n}) \|^2_{H^1(\O)} + \| \pi\|^2_{H^1(\O)} + \| \chi\|^2_{H^1(\O)}\right) \\
\leq  (1+2  \cter{cte:piH1.noreg.1} h) \left[\Ff(\bX^\star) + \int_\O \gravmech(\bphi) \cdot (\bu - \bu^\star)\right]  + \sum_{\a \in \{n,w\}} \int_\O (\phi_{\a} - \phi_\a^\star)\, p_\a^D + \cter{cte:NRG.nreg.2} {h}.
\end{multline}
\end{prop}
\begin{proof}
Let $\left(\eps_k\right)_{k\geq 0}\subset(0,\frac14]$ be a sequence of positive regularization parameters tending to $0$ as $k$ tend to $+\infty$, then we investigate, thanks to compactness arguments, the behavior when $k$ goes to $\infty$ of the sequence $\left(\bX_{\eps_k}\right)_{k\geq 0}$ provided by Proposition~\ref{prop:Schauder}. 
For the ease of reading, we will omit the subscript $k$, and we will denote by $\eps \to 0$ the limit $k\to +\infty$. 

We deduce from Corollary~\ref{coro:noreg.NRG} that there exists $C$ not depending on $\eps$ such that 
\[
\| \xi_\eps(s_{n,\eps})\|_{H^1(\O)} \leq C, \qquad \|G_\eps(\phi_\eps)\|_{H^1(\O)} \leq C, 
\]
\[
\| \pi_\eps\|_{H^1(\O)} \leq C, \qquad \|\theta_\eps\|_{H^1(\O)} \leq C, \qquad 
\| \bu_\eps\|_{\bU} \leq C, 
\]
while $\bphi_\eps \in \Kk_\bphi$ and $0 \leq s_{\a,\eps} \leq 1$ holds true for all $\eps>0$. Here, we employ the same arguments as in the proof of Lemma~\ref{lem:noreg.NRG} and Corollary~\ref{coro:noreg.NRG}, regarding state-dependent terms on the right hand side. 
Then we infer from~\eqref{eq:constraint.schauder} that 
\[
\|\phi_\eps\|_{H^1(\O)} \leq C.
\]
As a consequence, there exists $\Xi, \pi, \theta, \phi$ and $\chi$ in $H^1(\O)$ such that, up to a subsequence, 
\be\label{eq:noreg.convH1w}
 \xi_\eps(s_{n,\eps}) \underset{\eps \to 0}\longrightarrow \Xi, \qquad 
  \pi_\eps \underset{\eps \to 0}\longrightarrow \pi, \qquad 
    \theta_\eps \underset{\eps \to 0}\longrightarrow \theta, \qquad 
   \phi_\eps \underset{\eps \to 0}\longrightarrow \phi \quad \text{and}\quad 
G_\eps(\phi_{\eps})   \underset{\eps \to 0}\longrightarrow \chi, 
\ee
as well as $\bu\in \bU$ such that, up to a subsequence,
\be\label{eq:noreg.convH1w.u}
\bu_\eps  \underset{\eps \to 0}\longrightarrow \bu, \qquad 
  \div \bu_\eps \underset{\eps \to 0}\longrightarrow \div \bu,
\ee
the above convergences holding almost everywhere in $\O$ as well as in the weak $H^1(\O)$ sense. One further readily deduces from the uniform convergence of $\xi_\eps$ towards $\xi$ that $\xi_\eps(s_n^D)$ converges weakly in $H^1(\O)$ towards $\xi(s_n^D)$, and that 
$\Xi - \xi(s_n^D) \in V$. 
Moreover, since $\bphi_\eps \in \Kk_\bphi$ and $0 \leq s_{\a,\eps} \leq 1$ for all $\eps>0$, 
there exists $\bphi \in L^\infty(\O)^2$ with $\bphi \in \Kk_\bphi$ a.e. in $\O$ and $s_{\a} \in L^\infty(\O)$ with $0 \leq s_n \leq 1$ and $s_w = 1 - s_n$  such that 
\be\label{eq:noreg.convbphi}
\bphi_\eps\underset{\eps \to 0}\longrightarrow \bphi \quad \text{and} \quad s_{\a,\eps} \underset{\eps \to 0}\longrightarrow s_\a
\ee
in the $L^\infty(\O)^2$- and $L^\infty(\O)$-weak-$\star$ senses respectively. 
Denote by 
\[
\O_\flat = \{x\in\O \; | \; \phi(x) = \phi_\flat\}, \qquad \O^\sharp =  \{x\in\O \; | \; \phi(x) = \phi^\sharp\} \quad \text{and}\quad 
\Oo = \O \setminus(\O_\flat \cup \O_\sharp), 
\]
and let $x \in \Oo$ such that $\phi_\eps(x) \to \phi(x)$ and $G_\eps(\phi_\eps(x)) \to \chi(x)$ as $\eps \to 0$. Then for all $\eta>0$,  for $\eps$ smaller than some $\eps^\star$ depending on $\eta$ and on $x$, $\phi_\eps(x) \in (\phi_\flat + \eta, \phi^\sharp-\eta)$. 
As a consequence, 
\[
G_\eps(\phi_\flat + \eta) \leq G_\eps(\phi_\eps(x)) \leq G_\eps(\phi^\sharp - \eta),  
\]
and $G_\eps(\phi_\eps(x))$ goes to $0=\chi(x)$ thanks to~\eqref{eq:Gbeta.pointwise}. On the other hand, for $x \in \O_\flat$ such that 
$\phi_\eps(x) \to \phi_\flat$ and $G_\eps(\phi_\eps(x)) \to \chi(x)$, then for $\eps$ small enough, $\phi_\eps(x) \leq \frac{\phi_\flat + \phi^\sharp}2$, so that $G_\eps(x) \leq 0$. Therefore, $\chi(x) \leq 0$. Similarly, $\chi(x) \geq 0$ for almost all $x \in \O^\sharp$. Finally, we get that 
$\chi \in \ov \chi(\phi)$ a.e. in $\O$. 

Let us now show that $\Xi = \xi(s_n)$ and that the convergence~\eqref{eq:noreg.convbphi} holds almost everywhere in $\O$. Let $s \in L^\infty(\O)$ with $0 \leq s \leq 1$ be arbitrary, then we infer from the aforementioned convergences that 
\[
(\xi_\eps(s_{n,\eps}) - \xi(s)) (s_{n,\eps} - s) \underset{\eps \to 0}\longrightarrow (\Xi - \xi(s)) (s_n - s) 
\quad \text{weakly in }L^2(\O). 
\]
On the other hand, one has 
\[
(\xi_\eps(s_{n,\eps}) - \xi(s)) (s_{n,\eps} - s) = (\xi_\eps(s_{n,\eps}) - \xi(s_{n,\eps})) (s_{n,\eps} - s) + 
(\xi(s_{n,\eps}) - \xi(s)) (s_{n,\eps} - s). 
\]
The first term in the above right-hand side tends uniformly to $0$ because of the uniform convergence of $\xi_\eps$ towards $\xi$ which can be deduced from Assumption~(H\ref{H.gamma}) on $\gamma$ and of the uniform boundedness of $s_{n,\eps}$, whereas the second term is non-negative because $\xi$ is non-decreasing. Therefore, we obtain that 
\[(\Xi - \xi(s)) (s_n - s) \geq 0\quad \text{a.e. in $\O$ for any $s\in [0,1]$}.\] 
This implies that $\Xi = \xi(s_n)$ (see for instance \cite{ACM17}). Next, 
\[
|\xi(s_{n,\eps}) - \xi(s_n)| \leq |\xi(s_{n,\eps}) - \xi_\eps(s_{n,\eps})| + |\xi_\eps(s_{n,\eps}) - \xi(s_n)|. 
\]
The uniform convergence of $\xi_\eps$ towards $\xi$ implies that the first term in the right-hand side goes to $0$ with $\eps$ uniformly on $\O$. The second term converges to $0$ almost everywhere owing to \eqref{eq:noreg.convH1w}. As a consequence, $\xi(s_{n,\eps})$ tends almost everywhere to $\xi(s_n)$, and since $\xi^{-1}$ is continuous, one gets that 
\be\label{eq:noreg.convsn}
s_{n,\eps} \underset{\eps \to 0}\longrightarrow  s_n \quad \text{a.e. in $\O$}. 
\ee
Bearing in mind the almost everywhere convergence of $\phi_\eps$ towards $\phi$, then 
\[
\phi_{n,\eps} = s_{n,\eps} \phi_\eps \underset{\eps \to 0}\longrightarrow s_n \phi = \phi_n \quad \text{a.e. in $\O$}. 
\]
 
The convergence properties \eqref{eq:noreg.convH1w}, \eqref{eq:noreg.convH1w.u} and \eqref{eq:noreg.convbphi} are enough to pass to the limit in the linear relations to recover~\eqref{eq:noreg.lim.0} and \eqref{eq:noreg.lim.mecha}. 
Moreover, the energy estimate~\eqref{eq:noreg.lim.NRG} holds true thanks to Corollary~\ref{coro:noreg.NRG}, to the lower semi-continuity of 
$\Ff$ (see for instance \cite[Lemma~C.6]{kangourou} for the capillary energy part) and the lower-semi continuity of the squared norms for the weak convergences. So our last focus is on the proof of~\eqref{eq:noreg.lim.cons.n}, obtaining~\eqref{eq:noreg.lim.cons.w} being similar.

Starting from~\eqref{eq:schauder.cons}, one has that 
\be\label{eq:noreg.cons.eps}
\int_\O \frac{\phi_{n,\eps} - \phi_n^\star}{h} v + \int_\O \frac{k_\eps(s_{n,\eps})}{\mu_n} \bbK( \phi_\eps) \grad \left((\hat p_{n,\eps} + \pi_\eps + G_\eps(\phi_\eps)) - \rho_n \g \cdot (\bx + \bu_\eps)\right) \cdot \grad v = 0 \quad \text{for all}\; v \in V.
\ee 
Fix $v \in V$, then because of~\eqref{eq:noreg.convbphi} we readily get that 
\be\label{eq:noreg.cons.1}
\int_\O \frac{\phi_{n,\eps} - \phi_n^\star}{h} v\;  \underset{\eps \to 0}\longrightarrow \; \int_\O \frac{\phi_{n} - \phi_n^\star}{h} v.
\ee
One also easily shows, thanks to Assumption~(H\ref{H.K}), to~\eqref{eq:noreg.convH1w}, \eqref{eq:noreg.convsn} and  Lebesgue's dominated convergence theorem that $k_\eps(s_{n,\eps}) \bbK(\phi_\eps)\grad v$ converges (strongly) in $L^2(\O)^d$ towards $s_n \bbK(\phi)\grad v$.  Then we infer from~\eqref{eq:noreg.convH1w} that 
\begin{align}
\label{eq:noreg.cons.2}
\int_\O \frac{k_\eps(s_{n,\eps})}{\mu_n} \bbK( \phi_\eps) \grad \left( (\pi_\eps + G_\eps(\phi_\eps)) -\rho_n \g \cdot (\bx  + \bu_\eps)\right) \cdot \grad v \;  \\
\underset{\eps \to 0}\longrightarrow \; \int_\O \frac{s_n}{\mu_n} \bbK(\phi) \grad \left((\pi + \chi) -\rho_n \g \cdot (\bx + \bu) \right) \cdot \grad v. \nn
\end{align}
The last term to be studied is 
\[
\Jj_\eps = \int_\O \frac{k_\eps(s_{n,\eps})}{\mu_n} \bbK( \phi_\eps) \grad \hat p_{n,\eps} \cdot \grad v.
\]
Recall that $\hat p_{n,\eps} \in \ov p_{n,\eps}(s_{n,\eps})$, with 
\[
\ov p_{n,\eps}(s) = \begin{cases}
(-\infty, \gamma_\eps(0) + \gamma_\eps'(0)] & \text{if}\; s = 0, \\
f_\eps(s) := \gamma_\eps(s) + (1-s) \gamma_\eps'(s) & \text{if}\; s \in (0,1], 
\end{cases}
\]
the function $f_\eps$ being continuous and increasing. Therefore, $s_{n,\eps}$ is a continuous function of $\hat p_{n,\eps}$, i.e. 
$s_{n,\eps} = \hat \Ss_\eps(\hat p_{n,\eps})$ with 
\[
\hat \Ss_\eps(\hat p) = \begin{cases}
0 & \text{if}\; \hat p \leq \gamma_\eps(0) + \gamma_\eps'(0), \\
f_\eps^{-1}(\hat p) & \text{otherwise}, 
\end{cases}
\]
hence $\Jj_\eps$ can be rewritten as 
\[
\Jj_\eps = \int_\O \frac1{\mu_n} \bbK(\phi_\eps) \grad \hat \psi_\eps(\hat p_{n,\eps}) \cdot \grad v, \qquad \text{with\; $\hat \psi_\eps( \hat p)  = \int_{f_\eps(0)}^{\hat p} k_{\eps}(\hat \Ss_\eps(a)) \d a$}. 
\]
Define also the nonlinear function $\psi_\eps: [0,1] \to \R_+$ by 
\[\psi_\eps(s) =  \hat \psi_\eps(f_\eps(s))= \int_0^s k_{\eps}(a)(1-a) \gamma_\eps''(a) \d a, \qquad s \in [0,1].\]
It converges uniformly on $[0,1]$ towards $\psi$ defined in \eqref{eq:Kirchhoff.def} and therefore it is uniformly bounded in $[0,1]$ thanks to Assumption~(H\ref{H.gamma}). 
The function 
\[
\psi_\eps(s_{n,\eps}) - \hat \psi_\eps(\hat p_{n,\eps}) = 
\eps (\gamma_\eps(0) + \gamma_\eps'(0) - \hat p_{n,\eps} )^+
\]
vanishes on $\{s_{n,\eps} >0\}$ and, thus, owing to estimate~\eqref{eq:noreg.NRG}, it satisfies 
\be\label{eq:nonreg.errPsi}
\| \grad \big(\psi_\eps(s_{n,\eps}) - \hat \psi_\eps(\hat p_{n,\eps})\big)\|_{L^2(\O)^d}^2 
\leq C\eps^2 \| \grad \hat p_{n,\eps} \|_{L^2(\O)^d}^2
\leq C \eps
\ee 
for some $C$ possibly depending on $h>0$ but not on $\eps$. Moreover, since $0 \leq s_{n,\eps} \leq k_\eps(s_{n,\eps}) \leq 1$ and $k_\eps(s_{n,\eps}) \leq s_{n,\eps} + \eps$, and recalling~\eqref{eq:Kirchhoff.2}, we infer from Proposition~\ref{prop:noreg.NRG} that   
\begin{align*}
\| \grad \psi_\eps(s_{n,\eps})\|_{L^2(\O)^d}^2 \leq&\; \| \grad \hat \psi_\eps(\hat p_{n,\eps}) \|_{L^2(\O)^d}^2 \\
\leq &\;\int_\O k_\eps(s_{n,\eps}) |\grad \hat p_{n,\eps}|^2 \leq \| \grad \xi_\eps(s_{n,\eps}) \|_{L^2(\O)^d}^2 +
\eps  \| \grad \hat p_{n,\eps} \|_{L^2(\O)^d}^2 \leq C. 
\end{align*}
This implies that there exists some $\Psi \in H^1(\O)$ such that (up to a subsequence as usual)
\[
\psi_\eps(s_{n,\eps}) \underset{\eps\to0}\longrightarrow \Psi \quad \text{weakly in $H^1(\O)$ and almost everywhere}.
\]
Then because of~\eqref{eq:noreg.convsn} and of the uniform convergence of $\psi_\eps$ towards $\psi$, we can identify $\Psi$ as $\psi(s_n)$ thanks to arguments similar to those used previously for showing that $\Xi=\xi(s_n)$. 
Finally, \eqref{eq:nonreg.errPsi} shows that
\[
\grad \hat \psi_\eps(\hat p_{n,\eps}) \underset{\eps \to 0} \longrightarrow \grad \psi(s_n) \quad \text{weakly in } L^2(\O)^d. 
\]
Since $\bbK(\phi_\eps) \grad v$ converges strongly in $L^2(\O)^d$ towards $\bbK(\phi) \grad v$ thanks to Assumption~(H\ref{H.K}), we can pass to the limit in $\Jj_\eps$
\be\label{eq:noreg.cons.3}
\Jj_\eps \underset{\eps\to0}\longrightarrow \int_\O \frac1{\mu_n} \bbK(\phi) \grad \psi(s_n) \cdot \grad v. 
\ee
The combination of~\eqref{eq:noreg.cons.1}, \eqref{eq:noreg.cons.2} and \eqref{eq:noreg.cons.3} in \eqref{eq:noreg.cons.eps} gives \eqref{eq:noreg.lim.cons.n}, and ends the proof of Proposition~\ref{prop:noreg.lim}.
\end{proof}

\section{Concluding the proof of Theorem~\ref{thm:main}}\label{sec:time}

Let $\bX^0 = (\bphi^0, \bu^0, \theta^0, \pi^0 = M\theta^0)$ be as in~(H\ref{H.init}) and fix $h>0$. Then applying Proposition~\ref{prop:noreg.lim} recursively, we get sequences ${(\bX^j)}_{j\geq 0}$ and ${(\chi^j)}_{j \geq 1}$ of measurable functions such that, for all $j \geq 1$, 
$\bphi^j=(\phi_{n}^j, \phi_w^j) \in \Kk_\bphi$ a.e. in $\O$,  with $\phi^j = \phi_n^j + \phi_w^j\in H^1(\O)$ and $\xi(s_n^j) - \xi(s_n^D) \in V$  setting $s_n^j = \phi_n^j/\phi^j$, such that $\bu^j \in \bU$, such that  $\theta^j, \pi^j, \chi^j \in H^1(\O)$ with $\chi^j \in \ov \chi(\phi^j)$ a.e. in $\O$, fulfilling 
\be\label{eq:j.0}
\pi^j = M\theta^j \quad \text{and}\quad \phi^j - b \, \div \bu^j - \theta^j = \phi_r, 
\ee
such that 
\begin{alignat}{3}
\label{eq:j.cons.n}
\int_\O \frac{\phi_n^j - \phi_n^{j-1}}h w - \int_\O \frac{1}{\mu_n} \bbK(\phi^j) \left( \grad \psi(s_n^j) + s_n^j \grad \left((\pi^j +\chi^j) - \rho_n \g \cdot (\bx + \bu^j) \right) \right) \cdot \grad w &= 0, &\qquad &\forall w \in V,\\ 
\label{eq:j.cons.w}
\int_\O \frac{\phi_w^j - \phi_w^{j-1}}h w - \int_\O \frac{1}{\mu_w} \bbK(\phi^j) \left( - \grad \psi(s_n^j) + s_w^j \grad \left((\pi^j+\chi^j) -  \rho_w \g \cdot (\bx + \bu^j) \right) \right) \cdot \grad w &= 0, &\qquad& \forall w \in V,  
\end{alignat}
 and
\be\label{eq:j.mecha} 
\int_\O \left( 2 \mu \, \beps(\bu^j) : \beps(\bv) + \lambda (\div \bu^j)(\div \bv)\right) = \int_\O b \pi^j \, \div\bv + \int_\O \gravmech(\bphi^j) \cdot \bv, \qquad \forall \bv \in V^d.
\ee
Moreover, we deduce from the energy estimate~\eqref{eq:noreg.lim.NRG} that 
\begin{multline}\label{eq:j.NRG}
\Ff(\bX^j) 
+ \cter{cte:NRG.nreg.1} h \left( \|\xi(s_{n}^j) \|^2_{H^1(\O)} + \| \pi^j \|^2_{H^1(\O)} + \| \chi^j \|^2_{H^1(\O)}\right) \\
\leq  (1+2  \cter{cte:piH1.noreg.1}h) \left[\Ff(\bX^{j-1}) + \int_\O \gravmech(\bphi^j) \cdot (\bu^j - \bu^{j-1}) \right] + \sum_{\a \in \{n,w\}} \int_\O (\phi_{\a}^j - \phi_\a^{j-1})\, p_\a^D + \cter{cte:NRG.nreg.2} {h}
\end{multline}
with constants $\cter{cte:piH1.noreg.1}$, $\cter{cte:NRG.nreg.1}$ and $\cter{cte:NRG.nreg.2}$ depending neither on $j$ nor on $h$, which in particular also provides the uniform bound for the phase porosity increments in the dual norm
\be\label{eq:j.dtphi}
\|\phi_\a^j - \phi_\a^{j-1} \|_{V'} \leq \cter{cte:bound_dtphi} h \quad \forall \a \in\{n,w\}\text{ and }j\geq 1,
\ee
for a positive constant $\ctel{cte:bound_dtphi}$ not depending on $j$ and $h$.

We define the piecewise constant in time approximations $\bX_h = (\bphi_h = (\phi_{n,h}, \phi_{w,h}), \bu_h, \theta_h, \pi_h)$ and $\chi_h$ by setting 
\[
\bX_h(t,x) = \bX^j(x), \quad \chi_h(t,x) = \chi^j(x) \quad \text{for}\; (t,x) \in \big((j-1)h, jh\big] \times \O. 
\]
We also denote by $\phi_h = \phi_{n,h}+\phi_{w,h}$ and $s_{\a,h} = \phi_{\a,h} / \phi_h$ for $\a\in\{n,w\}$. We extend $\bX_h$ to negative times by setting 
$\bX_h(t) = \bX^0$ for $t\leq 0$. Moreover, we define the approximate time derivatives $\p_t^h \phi_{\a,h}$, $\a\in\{n,w\}$, and $\p_t^h \phi_{h}$ by 
\[
\p_t^h \phi_{\a,h}(t,\cdot) = \frac{\phi_\a^j - \phi_\a^{j-1}}h \qquad \text{for}\; t \in \big((j-1)h, jh\big), \qquad \text{and}\quad  \p_t^h \phi_{h} = \p_t^h \phi_{n,h}+\p_t^h \phi_{w,h}.
\]

We will now let $h$ tend to $0$. Once again, our proof relies on compactness arguments. In what follows, the limit $h\to0$ implicitly refers to the convergence when $\ell$ goes to $+\infty$ of a decreasing sequence $\left(h_\ell\right)_{\ell\geq 0}$ of positive times steps tending to $0$.

Let $T>0$ be an arbitrary finite time horizon, then summing~\eqref{eq:j.NRG} for $j=1,\dots, \lceil \frac Th \rceil$ leads to 
\begin{align}
\label{eq:pre_gronwall.j.1}
\Ff(\bX_h(T)) + &\; \cter{cte:NRG.nreg.1} \left( \|\xi(s_{n,h})\|_{L^2((0,T);H^1(\O))}^2  +  \|\pi_h \|_{L^2((0,T);H^1(\O))}^2 + \|\chi_h\|_{L^2((0,T);H^1(\O))}^2\right) \\
& \leq \Ff(\bX^0) +2  \cter{cte:piH1.noreg.1} \sum_{j=1}^{\lceil \frac Th \rceil} h \Ff(\bX^{j-1}) + (1 + 2h\cter{cte:piH1.noreg.1}) \sum_{j=1}^{\lceil \frac Th \rceil}\int_\O \gravmech(\bphi^j) \cdot (\bu^j - \bu^{j-1})
\nn
\\
&\qquad + \sum_{\a \in \{n,w\}} \int_\O (\phi_{\a,h}(T) - \phi_\a^0)\, p_\a^D +  \cter{cte:NRG.nreg.2}(T+1) 
\nn
\end{align}
Employing summation by parts, used that the summation can be extended to index $j=0$ corresponding to zero due to the above extension to negative times, utilizing that $\gravmech$ is uniformly bounded and affine due to (H\ref{H.gravity}), and employing the uniform bound~\eqref{eq:j.dtphi}, we obtain for a suitable constant $C>0$ that
\begin{align}
 \label{eq:pre_gronwall.j.2}
 &\sum_{j=1}^{\lceil \frac Th \rceil} \int_\O \gravmech(\bphi^j) \cdot (\bu^j - \bu^{j-1}) \\
&\qquad= \int_\O \left[ \gravmech(\bphi_h(T)) \cdot \bu_h(T) - \gravmech(\bphi^0) \cdot \bu^0 \right] - \sum_{j=1}^{\lceil \frac Th \rceil} \int_\O \left(\gravmech(\bphi^j) - \gravmech(\bphi^{j-1})\right) \cdot \bu^{j-1} 
\nn
 \\
 &\qquad\leq \int_\O \left[ \gravmech(\bphi_h(T)) \cdot \bu_h(T) - \gravmech(\bphi^0) \cdot \bu^0 \right] + \sum_{j=1}^{\lceil \frac Th \rceil} \sum_{\a \in \{n, w\}} \rho_a
 \left\| \phi_\a^j - \phi_\a^{j-1}  \right\|_{V'} \, \| \g \cdot \bu^{j-1} \|_V 
 \nn \\
 &\qquad\leq C  + \frac{1}{4} \Ff(X_h(T)) + \cter{cte:piH1.noreg.1} \sum_{j=1}^{\lceil \frac Th \rceil} h  \Ff(\bX^{j-1}).
 \nn
\end{align}
Furthermore, using the uniform boundedness of $\bphi$ and of $\bp^D$ results in
\be
\label{eq:pre_gronwall.j.3}
\sum_{\a \in \{n,w\}} \int_\O (\phi_{\a}^j - \phi_\a^{j-1})\, p_\a^D \leq C' 
\ee
for a suitable constant $C'>0$. Thus, combining~\eqref{eq:pre_gronwall.j.1}--\eqref{eq:pre_gronwall.j.3}, and assuming $2h\cter{cte:piH1.noreg.1} \leq 1$ results in
\begin{align}
 \label{eq:pre_gronwall.j.2}
 \frac12 \Ff(\bX_h(T)) + &\; \cter{cte:NRG.nreg.1} \left( \|\xi(s_{n,h})\|_{L^2((0,T);H^1(\O))}^2  +  \|\pi_h \|_{L^2((0,T);H^1(\O))}^2 + \|\chi_h\|_{L^2((0,T);H^1(\O))}^2\right) \\
&\leq \Ff(\bX^0) + 4  \cter{cte:piH1.noreg.1} \sum_{j=1}^{\lceil \frac Th \rceil} h \Ff(\bX^{j-1}) + C, \nn
\end{align}
for $C$ depending on $T$ but not on $h$.
Then we deduce from a discrete Gronwall lemma that, for all $T\geq0$, there holds
\begin{multline}\label{eq:j.NRG.2}
\Ff(\bX_h(T)) + \|\xi(s_{n,h})\|_{L^2((0,T);H^1(\O))}^2  +  \|\pi_h \|_{L^2((0,T);H^1(\O))}^2 + \|\chi_h\|_{L^2((0,T);H^1(\O))}^2\leq C (1+  \Ff(\bX^0)),
\end{multline}
where $C$ depends on $T$ but not on $h$. Assumption~(H\ref{H.init}) ensures that $\Ff(\bX^0)$ is finite, 
whereas we deduce from Assumption~(H\ref{H.Omega}), from \eqref{eq:j.NRG.2} and Korn's inequality that 
\be\label{eq:uh.L2H2}
\|\bu_h\|_{L^\infty((0,T);H^1(\O)^d)}+ \|\grad (\div\bu_h)\|_{L^2((0,T)\times\O)^d} \leq C. 
\ee
Then we deduce from the constraints~\eqref{eq:j.0} that 
\be\label{eq:phih.L2H1}
\|\theta_h\|_{L^2((0,T);H^1(\O))} \leq C, \qquad \|\phi_h\|_{L^2((0,T);H^1(\O))} \leq C.
\ee
Since $\psi\circ \xi^{-1}$ is $\frac12$-Lipschitz continuous, one further deduces from~\eqref{eq:j.NRG.2} that 
\be\label{eq:psi.L2H1}
\|\psi(s_{n,h})\|_{L^2((0,T);H^1(\O))} \leq C. 
\ee

Let $v: \R_+ \times \ov \O \to \R$ be a smooth function such that $v(t,x)=0$ for $(t,x) \in \R_+ \times \G^D$ as well as for $t \geq T$. 
For $h>0$ fixed, we define $\left(v^j\right)_{j\geq0} \subset V$ by 
\[
v^j = \frac1h\int_{(j-1)h}^{jh} v(t,\cdot) \d t \quad\text{for}\;  j \geq 1 \qquad\text{and}\quad v^0 = v(0,\cdot), 
\]
and by 
\[
v_h(t,x) = v^j (x)  \quad \text{for}\; (t,x) \in \big((j-1)h, jh\big] \times \ov\O. 
\]
Choosing $w = h v^j$ in~\eqref{eq:j.cons.n} and summing over $j \geq 1$ gives 
\be\label{eq:weak_h}
\iint_{\R_+\times\O} \p_t^h \phi_{n,h} v = \iint_{\R_+\times\O}  \frac{1}{\mu_n} \bbK(\phi_h) \left( \grad \psi(s_{n,h}) + s_{n,h} (\grad (\pi_h +\chi_h) - \rho_n \g \cdot (\bx + \bu_h) ) \right) \cdot \grad v. \ee
Then we infer from~\eqref{eq:j.NRG.2} and \eqref{eq:psi.L2H1} together with $0 \leq s_{n,h} \leq 1$ and the uniform boundedness of $\bbK(\phi_h)$ deduced from Assumption~(H\ref{H.K}) that 
\be\label{eq:ptphi_a.L2V'}
\left\|\p_t^h \phi_{n,h} \right\|_{L^2((0,T);V')} \leq C, \qquad \left\|\p_t^h \phi_{w,h} \right\|_{L^2((0,T);V')} \leq C,
\ee
the second estimate being similar, and thus that 
\be\label{eq:ptphi.L2V'}
\left\|\p_t^h \phi_{h} \right\|_{L^2((0,T);V')} \leq C.
\ee

We infer from~\eqref{eq:j.NRG.2} and~\eqref{eq:phih.L2H1} that there exists $\Xi, \pi,\chi$ and $\theta$ in $L^2((0,T);H^1(\O))$ such that, up to the extraction of a subsequence, there holds 
\begin{align}
\label{eq:limh.xi}\xi(s_{n,h}) \underset{h\to0} \longrightarrow \Xi \quad& \text{weakly in}\; L^2((0,T);H^1(\O)),\\
\label{eq:limh.pi}\pi_h \underset{h\to0} \longrightarrow \pi \quad& \text{weakly in}\; L^2((0,T);H^1(\O)), \\
\label{eq:limh.chi}\chi_h \underset{h\to0} \longrightarrow \chi \quad& \text{weakly in}\; L^2((0,T);H^1(\O)), \\
\label{eq:limh.theta}\theta_h \underset{h\to0} \longrightarrow \theta \quad& \text{weakly in}\; L^2((0,T);H^1(\O)), 
\end{align}
whereas~\eqref{eq:uh.L2H2} ensures the existence of some $\bu \in L^\infty((0,T);H^1(\O)^d)\cap L^2((0,T); \bU)$ such that, up to a subsequence, 
\be\label{eq:limh.u}
\bu_h  \underset{h\to0} \longrightarrow \bu \quad \text{in the $L^\infty((0,T);H^1(\O)^d)$-weak-$\star$ and $L^2((0,T);\bU)$-weak sense}. 
\ee
Since $\bphi_h$ belongs to the bounded convex subset $\Kk_\bphi$ a.e. in $\R_+\times \O$, we get that 
\be\label{eq:limh.bphi}
\bphi_h  \underset{h\to0} \longrightarrow \bphi \quad \text{in the}\; L^\infty(\R_+\times\O)^2\text{-weak-$\star$ sense}
\ee
with $\bphi(t,x) \in \Kk_\bphi$ for a.e. $(t,x) \in \R_+ \times \O$. We denote by $\phi = \phi_n+\phi_w$, so that, owing to~\eqref{eq:phih.L2H1}, 
\be\label{eq:limh.phi}
\phi_h  \underset{h\to0} \longrightarrow \phi \quad \text{in the}\; L^\infty(\R_+\times\O)\text{-weak-$\star$ and $L^2((0,T);H^1(\O))$-weak senses}, 
\ee
and $\phi_\flat \leq \phi \leq \phi^\sharp$ almost everywhere. 
Note that at this point, we have already enough material to pass to the limit in the linear relations
\[
\pi_h = M \theta_h, \qquad \phi_h - b \div \bu_h - \theta_h = \phi_r
\]
and 
\[
\iint_{(0,T) \times \O} 2 \mu \, \beps(\bu_h): \beps(\bv) + \lambda (\div \bu_h)(\div\bv) = \iint_{(0,T) \times \O} b\, \pi_h\, \div \bv + \int_\O \gravmech(\bphi_h) \cdot \bv, \qquad \bv \in L^2((0,T);V^d)
\]
to recover \eqref{eq:3.theta}, \eqref{eq:constraint}, as well as \eqref{eq:weak.u}. What remains in this section is devoted to the proof of the weak forms~\eqref{eq:weak.n} and \eqref{eq:weak.w} of the phase volume conservation equations~\eqref{eq:1}.

We deduce from~\eqref{eq:ptphi_a.L2V'} the existence of $\delta_n, \delta_w \in L^2((0,T); V')$ and of $\delta = \delta_n+ \delta_w$ such that 
\[
\p_t^h \phi_{\a,h} \underset{h\to0} \longrightarrow \delta_\a \quad \text{and}\quad \p_t^h \phi_{h} \underset{h\to0} \longrightarrow \delta \quad \text{weakly in}\; L^2((0,T);V'),\quad \a \in \{n,w\}. 
\]
Let $\varphi \in C^\infty_c((0,T)\times \O)$, and assume that $\varphi(t) = 0$ for $t \leq h$ and $t\geq T-h$, which is not a restriction since $h$ is going to tend to $0$. 
Then for $\a \in \{n,w\}$, there holds
\[
\iint_{(0,T) \times \O} \p_t^h \phi_{\a,h}\, \varphi = - \iint_{(0,T) \times \O} \phi_{\a,h} \tilde \p_t^h \varphi 
\]
where we have set 
\[
 \tilde \p_t^h \varphi(t,x) = \frac1h \int_{(j-1)h}^{jh} \frac{\varphi(\tau+h,x) - \varphi(\tau,x)}h \d \tau \quad\text{if} \; (t,x) \in \big((j-1)h, jh\big) \times \O.
\]
Since $\varphi$ is assumed to be regular, one readily checks that $ \tilde \p_t^h \varphi$ converges uniformly towards $\p_t \varphi$, hence 
\[
\iint_{(0,T) \times \O} \p_t^h \phi_{\a,h}\, \varphi \underset{h\to0} \longrightarrow \int_0^T \langle \delta_\a\, , \, \varphi \rangle_{V',V} = - \iint_{(0,T) \times \O} \phi_\a \p_t \varphi, 
\]
so that $\delta_\a = \p_t \phi_\a$, thus 
\be\label{eq:limh.ptphi}
\p_t^h \phi_{\a,h} \underset{h\to0} \longrightarrow \p_t \phi_\a \quad \text{and}\quad \p_t^h \phi_{h} \underset{h\to0} \longrightarrow \p_t \phi  \quad \text{weakly in}\; L^2((0,T);V'),\quad \a \in \{n,w\}. 
\ee
We further deduce from \eqref{eq:limh.phi} and \eqref{eq:limh.ptphi} that one can apply a discrete (here in time only) Aubin-Simon type result (see for instance~\cite{GL12} or \cite{ACM17}), to get that, up to a subsequence, 
\be\label{eq:limh.phi.2}
\phi_h \underset{h\to0}\longrightarrow \phi \quad\text{a.e. in }(0,T)\times \O. 
\ee

Denote by $s_n$ the $L^\infty((0,T)\times\O)$-weak-$\star$ limit of $s_{n,h}$, then the convergences \eqref{eq:limh.bphi} and \eqref{eq:limh.phi} show that $s_n = \phi_n/\phi$. In order to show that $s_{n,h}$ converges pointwise towards $s_n$, we take inspiration on what is done in~\cite{BBDM21}. Let $f:[0,1] \to \R$ be a continuous increasing function such that $f\circ\psi^{-1}$ and $f$ are both Lipschitz continuous, so that, thanks to~\eqref{eq:psi.L2H1}, there holds  
\be\label{eq:f.L2H1}
\| f(s_{n,h}) \|_{L^2((0,T);H^1(\O))} \leq C
\ee 
for some $C$ not depending on $h$. 
Let $\tau>0$ and $y \in \R^d$, and denote by $\O_y = \{x \in \O \; | \; [x,y] \subset \O\}$, then 
\[
\iint_{(0,T-\tau) \times \O_y} \big|f(s_{n,h}(t+\tau, x+y)) - f(s_{n,h}(t,x)) \big|^2\leq 2 (\Aa_h^\tau + \Bb_h^y)
\]
with 
\begin{align*}
\Aa_h^\tau  = &\; \iint_{(0,T-\tau) \times \O} \big|f(s_{n,h}(t+\tau, x)) - f(s_{n,h}(t,x)) \big|^2, \\
\Bb_h^y = & \; \iint_{(0,T) \times \O_y}  \big|f(s_{n,h}(t, x+y)) - f(s_{n,h}(t,x)) \big|^2.
\end{align*}
Because of \eqref{eq:f.L2H1} and of the classical characterization~\cite[Proposition 9.3]{Brezis11} of the space $H^1(\O)$, one can estimate $\Bb_h^y$ by 
\[
\Bb_h^y \leq C |y|^2
\]
for some $C$ not depending on $h$. 
Let us now adapt the ideas of~\cite{AL83} to the term $\Aa_\tau$. As $f$ is Lipschitz continuous with Lipschitz constant denoted by $L_f$, there holds 
\begin{align*}
\Aa_h^\tau \leq&\;  L_f  \iint_{(0,T-\tau) \times \O} \big(s_{n,h}(t+\tau) - s_{n,h}(t)\big) \big(f(s_{n,h}(t+\tau)) - f(s_{n,h}(t))\big) \\
\leq & \; \frac{L_f}{\phi_\flat}  \iint_{(0,T-\tau) \times \O} \big(\phi_{n,h}(t+\tau) - \phi_{n,h}(t)\big) \big(f(s_{n,h}(t+\tau)) - f(s_{n,h}(t))\big).
\end{align*}
Bearing in mind the definition of $\phi_{n,h}$, we get that 
\[
\Aa_h^\tau \leq \frac{L_f}{\phi_\flat}  \iint_{(0,T-\tau) \times \O} \sum_{j =\lceil \frac th \rceil +1}^{\lceil \frac{t+\tau}h \rceil} \left( \phi_n^j - \phi_n^{j-1} \right) \big(f(s_{n,h}(t+\tau)) - f(s_{n,h}(t))\big).
\]
Then using~\eqref{eq:j.cons.n} leads to 
\[
\Aa_h^\tau \leq h \frac{L_f}{\phi_\flat}  \iint_{(0,T-\tau) \times \O}  \sum_{j =\lceil \frac th \rceil +1}^{\lceil \frac{t+\tau}h \rceil} \left|\bF_n^j\right|\,  \left|\grad \big(f(s_{n,h}(t+\tau)) - f(s_{n,h}(t))\big) \right|
\]
where the flux 
\[
\bF_n^j = -\frac1\mu_n \bbK(\phi^j)\left( \grad \psi(s_n^j) + s_{n}^j  \grad \left((\pi^j + \chi^j) - \rho_n \g \cdot(\bx + \bu^j) \right)\right)
\]
satisfies 
\[
\int_\O \sum_{j=1}^{\lceil \frac T\tau \rceil} h \left| \bF_n^j \right|^2 \leq C
\]
in view of the a priori estimates \eqref{eq:j.NRG.2} and \eqref{eq:psi.L2H1}. Applying the Cauchy-Schwarz inequality gives 
\[
\Aa_h^\tau \leq \sqrt{\Aa_h^{\tau,(1)} \Aa_h^{\tau,(2)}},
\]
where
\[
\Aa_h^{\tau,(1)} = 2  \frac{L_f}{\phi_\flat}  \left(\sum_{j =\lceil \frac th \rceil +1}^{\lceil \frac{t+\tau}h \rceil} h \right) \| f(s_{n,h}) \|_{L^2((0,T);H^1(\O))}^2\overset{\eqref{eq:f.L2H1}}\leq C (\tau + h),
\]
and where 
\[
\Aa_h^{\tau,(2)} \leq  \frac{L_f}{\phi_\flat} \iint_{(0,T-\tau) \times \O}  \sum_{j =\lceil \frac th \rceil +1}^{\lceil \frac{t+\tau}h \rceil} h \left| \bF_n^j \right|^2 \leq C \tau.
\]
As a consequence, we can use the Riesz-Fréchet-Kolmogorov compactness criterion~\cite[Theorem 4.26]{Brezis11} to claim that, up to a subsequence, 
\[
f(s_{n,h}) \underset{h\to0} \longrightarrow \mathfrak f \quad \text{a.e. in}\; (0,T)\times \O. 
\]
As $f$ is increasing, this implies that $s_{n,h}$ converges also almost everywhere, and by uniqueness of the limit we have that  
\be\label{eq:limh.sn.ae}
s_{n,h} \underset{h\to0} \longrightarrow s_n \quad \text{a.e. in}\; (0,T)\times \O. 
\ee
In combination with \eqref{eq:psi.L2H1}, this ensures that, up to a subsequence, there holds 
\be\label{eq:limh.psi}
\psi(s_{n,h}) \underset{h\to0} \longrightarrow \psi(s_n) \quad \text{a.e. and weakly in } L^2((0,T);H^1(\O)).
\ee
We can now pass to the limit in~\eqref{eq:weak_h}. Thanks to~\eqref{eq:limh.ptphi}, one has 
\[
\iint_{\R_+\times\O} \p_t^h \phi_{n,h} v   \underset{h\to0} \longrightarrow \int_{\R_+} \langle \p_t \phi_\a, v \rangle_{V',V}. 
\]
Moreover, the convergences~\eqref{eq:limh.phi.2} and~\eqref{eq:limh.sn.ae} together with Assumption (H\ref{H.K}) on $\bbK$ show that 
\[
\bbK(\phi_h) \grad v  \underset{h\to0} \longrightarrow  \bbK(\phi) \grad v
\quad 
\text{and}
\quad 
s_{n,h} \bbK(\phi_h) \grad v  \underset{h\to0} \longrightarrow  s_n \bbK(\phi) \grad v
\quad \text{strongly in}\; L^2(\R_+\times \O).
\]
Combining this with the weak convergences~\eqref{eq:limh.pi},\ \eqref{eq:limh.chi},\ \eqref{eq:limh.psi}, and~\eqref{eq:limh.u}, one gets that 
\begin{multline*}
 \iint_{\R_+\times\O}  \frac{1}{\mu_n} \bbK(\phi_h) \left( \grad \psi(s_{n,h}) + s_{n,h} \grad \left((\pi_h +\chi_h) - \rho_n \g \cdot (\bx + \bu_h)\right) \right) \cdot \grad v \\
  \underset{h\to0} \longrightarrow
  \iint_{\R_+\times\O}  \frac{1}{\mu_n} \bbK(\phi) \left( \grad \psi(s_{n}) + s_{n} \grad \left((\pi +\chi) - \rho_n \g \cdot (\bx + \bu) \right) \right) \cdot \grad v, 
\end{multline*}
so that \eqref{eq:weak_h} gives~\eqref{eq:weak.n} at the limit $h\to0$. Similar arguments allow to recover \eqref{eq:weak.w}. 

The last step to conclude the proof of Theorem~\ref{thm:main} is to show that $\chi \in \ov \chi(\phi)$. 
Combining \eqref{eq:ptphi.L2V'} with \eqref{eq:limh.chi} and~\eqref{eq:limh.phi}, we can apply \cite[Proposition 3.8]{ACM17} which gives that 
\be\label{eq:limh.moussa}
\iint_{(0,T)\times\O} \phi_h\, \chi_h \, \varphi  \; \underset{h\to0} \longrightarrow\;  \iint_{(0,T)\times\O} \phi\, \chi \, \varphi
\ee
for all $\varphi \in L^2((0,T)\times \O)$. 
This allows to use Minty's trick to show that $\chi \in \ov \chi(\phi)$ a.e. in $\O$. Indeed, $\chi_h \in \ov \chi(\phi_h)$ is equivalent to the fact that, for all 
$(a,A) \in [\phi_\flat, \phi^\sharp]\times \R$ such that $A \in \ov \chi(a)$, then 
$
(\phi_h - a) (\chi_h - A) \geq 0. 
$
Therefore, for all $\varphi \in L^2((0,T)\times \O)$ with $\varphi \geq 0$, one has 
\[
\iint_{(0,T)\times \O} (\phi_h - a) (\chi_h - A) \varphi \geq 0. 
\]
In view of~\eqref{eq:limh.chi}, \eqref{eq:limh.phi} and \eqref{eq:limh.moussa}, letting $h$ tend to $0$ gives
\[
\iint_{(0,T)\times \O} (\phi - a) (\chi - A) \varphi \geq 0, \qquad \forall \varphi \in L^2((0,T)\times \O), \; \varphi \geq 0. 
\]
This implies that $(\phi - a) (\chi - A) \geq 0$ a.e. in  $(0,T) \times \O$, and since $(a,A)$ is arbitrary in the graph $\ov \chi$, then 
$
\chi \in \ov \chi(\phi)$ a.e. in $(0,T)\times \O.
$ 
This concludes the proof of Theorem~\ref{thm:main}.

\section{Conclusion and prospects}\label{sec:conclusion}

In this paper we showed the first global existence for degenerate multiphase poromechanics under some reasonable assumption of the domain $\O$, i.e. $\O$ is homogeneous in space and is such that the elasticity equation enjoy some regularity (H\ref{H.Omega}). The model includes, among others, capillarity effects, nonlinear pore pressure constitutive law, porosity dependent permeability and gravitational forces, in a thermodynamically consistent manner. For simplicity, we have restricted our analysis to the case of linear relative permeabilities, even though our purpose should transpose to mild nonlinear relative permeabilities. Some weak-coupling condition between the flow and the mechanics is also postulated, cf. (H\ref{H.weak-coupling}). 
This assumption was necessary in our analysis to get estimates on $\grad \pi$ and $\grad \xi$ separately. The assumption then becomes needless in the case where $\phi_\flat < \phi < \phi^\sharp$, yielding $\chi = 0$, but we are not able to guaranty this situation a priori. The main a priori estimate is derived from the time evolution of the Helmholtz free energy, which is decreasing up to contributions coming from the boundary and from gravity.

The extension of our global existence result to more complex frameworks is an open problem. Mathematical difficulties have to be bypassed here, in particular, as the proofs of Propositions \ref{prop:Pregdis.NRG} and \ref{prop:noreg.NRG} strongly rely on the fact that the problem is essentially homogeneous (or at least smoothly varying)  in space and on the weak coupling assumption~(H\ref{H.weak-coupling}). Extending the proof of Brenner and Sung~\cite{BS92} to prove (H\ref{H.Omega}) is also a challenging problem which can be of interest in other contexts.

Finally, numerical experiments should be carried out in future works. The rigorous proof of the convergence of numerical methods could even be achieved thanks to compactness arguments adapting ours.

\subsection*{Acknowledgements}
JWB and CC warmly thank Konstantin Brenner for stimulating exchanges on multiphase poromechanics. CC also thanks Jérôme Droniou and Roland Masson for enlightening discussions. 
This work was done in the framework of the GradFlowPoro project supported by the Campus France (project N° 48356QE) and Research council of Norway (RCN project \#331960). CC also acknowledges partial support from Labex CEMPI (ANR-11-LABX-0007-01). JWB also acknowledges partial support from the UoB Akademia-project FracFlow.

\subsection*{Statements}
The authors have no conflict of interest related to this work. 
This manuscript has no associated data.


\begin{thebibliography}{10}

\bibitem{AL83}
H.~W. Alt and S.~Luckhaus.
\newblock Quasilinear elliptic-parabolic differential equations.
\newblock {\em Math. Z.}, 183(3):311--341, 1983.

\bibitem{altmann2022decoupling}
R.~Altmann and R.~Maier.
\newblock A decoupling and linearizing discretization for weakly coupled
  poroelasticity with nonlinear permeability.
\newblock {\em SIAM J. Sci. Comput.}, 44(3):B457--B478, 2022.

\bibitem{ambartsumyan2019nonlinear}
I.~Ambartsumyan, V.~J. Ervin, T.~Nguyen, and I.~Yotov.
\newblock A nonlinear stokes--biot model for the interaction of a non-newtonian
  fluid with poroelastic media.
\newblock {\em ESAIM Math. Model. Numer. Anal}, 53(6):1915--1955, 2019.

\bibitem{ACM17}
B.~Andreianov, C.~Canc\`es, and A.~Moussa.
\newblock A nonlinear time compactness result and applications to
  discretization of degenerate parabolic–elliptic {PDE}s.
\newblock {\em J. Funct. Anal.}, 273(12):3633--3670, 2017.

\bibitem{auriault1977sanchez}
J.-L. Auriault and E.~Sanchez-Palencia.
\newblock \'etude du comportement macroscopique d’un mileu poreux satur\'e
  deformable.
\newblock {\em J. Mecanique}, 16:575603, 1977.

\bibitem{BB90}
J.~Bear and Y.~Bachmat.
\newblock {\em Introduction to modeling of transport phenomena in porous
  media}.
\newblock {K}luwer {A}cademic {P}ublishers, {D}ordrecht, {T}he {N}etherlands,
  1990.

\bibitem{Biot41}
M.~A. Biot.
\newblock General theory of three-dimensional consolidation.
\newblock {\em J. Appl. Phys.}, 12(2):155--164, 1941.

\bibitem{bociu2021multilayered}
L.~Bociu, S.~Canic, B.~Muha, and J.~T. Webster.
\newblock Multilayered poroelasticity interacting with {S}tokes flow.
\newblock {\em SIAM J. Math. Anal.}, 53(6):6243--6279, 2021.

\bibitem{bociu2016analysis}
L.~Bociu, G.~Guidoboni, R.~Sacco, and J.~T. Webster.
\newblock Analysis of nonlinear poro-elastic and poro-visco-elastic models.
\newblock {\em Arch. Ration. Mech. Anal.}, 222:1445--1519, 2016.

\bibitem{bociu2023mathematical}
L.~Bociu, B.~Muha, and J.~T. Webster.
\newblock Mathematical effects of linear visco-elasticity in quasi-static
  {B}iot models.
\newblock {\em J. Math. Anal. Appl.}, page 127462, 2023.

\bibitem{BBDM21}
F.~Bonaldi, K.~Brenner, J.~Droniou, and R.~Masson.
\newblock Gradient discretization of two-phase flows coupled with mechanical
  deformation in fractured porous media.
\newblock {\em Comput. Math. Appl.}, 98:40--68, 2021.

\bibitem{boon2023mixed}
W.~M. Boon and J.~M. Nordbotten.
\newblock Mixed-dimensional poromechanical models of fractured porous media.
\newblock {\em Acta Mech.}, 234(3):1121--1168, 2023.

\bibitem{both2019anderson}
J.~W. Both, K.~Kumar, J.~M. Nordbotten, and F.~A. Radu.
\newblock Anderson accelerated fixed-stress splitting schemes for consolidation
  of unsaturated porous media.
\newblock {\em Comput. Math. Appl.}, 77(6):1479--1502, 2019.

\bibitem{both2019gradient}
J.~W. Both, K.~Kumar, J.~M. Nordbotten, and F.~A. Radu.
\newblock The gradient flow structures of thermo-poro-visco-elastic processes
  in porous media.
\newblock {\em arXiv preprint arXiv:1907.03134}, 2019.

\bibitem{both2021global}
J.~W. Both, I.~S. Pop, and I.~Yotov.
\newblock Global existence of weak solutions to unsaturated poroelasticity.
\newblock {\em ESAIM Math. Model. Numer. Anal.}, 55(6):2849--2897, 2021.

\bibitem{both2017robust}
Jakub~Wiktor Both, Manuel Borregales, Jan~Martin Nordbotten, Kundan Kumar, and
  Florin~Adrian Radu.
\newblock Robust fixed stress splitting for biot’s equations in heterogeneous
  media.
\newblock {\em Applied Mathematics Letters}, 68:101--108, 2017.

\bibitem{BJS09}
A.~Bourgeat, M.~Jurak, and F.~Sma\"{i}.
\newblock Two-phase, partially miscible flow and transport modeling in porous
  media; application to gas migration in a nuclear waste repository.
\newblock {\em Comput. Geosci.}, 13:29--42, 2009.

\bibitem{BCH13}
K.~Brenner, C.~Canc{\`e}s, and D.~Hilhorst.
\newblock {Finite volume approximation for an immiscible two-phase flow in
  porous media with discontinuous capillary pressure}.
\newblock {\em Comput. Geosci.}, 17(3):573--597, 2013.

\bibitem{BS92}
S.~C. Brenner and L.-Y. Sung.
\newblock Linear finite element methods for planar linear elasticity.
\newblock {\em Math. Comp.}, 59(200):321--338, 1992.

\bibitem{Brezis6566}
H.~Brezis.
\newblock Les op\'erateurs monotones.
\newblock {\em S\'eminaire Choquet. Initiation \`a l'analyse}, 5(2), 1965-1966.

\bibitem{Brezis73}
H.~Brezis.
\newblock {\em Op\'erateurs maximaux monotones et semi-groupes de contractions
  dans les espaces de {H}ilbert}.
\newblock North-Holland Publishing Co., 1973.

\bibitem{Brezis11}
H.~Brezis.
\newblock {\em Functional analysis, {S}obolev spaces and partial differential
  equations}.
\newblock Universitext. Springer, New York, 2011.

\bibitem{brun2019well}
M.~K. Brun, E.~Ahmed, J.~M. Nordbotten, and F.~A. Radu.
\newblock Well-posedness of the fully coupled quasi-static thermo-poroelastic
  equations with nonlinear convective transport.
\newblock {\em J. Math. Anal. Appl.}, 471(1-2):239--266, 2019.

\bibitem{BLS09}
F.~Buzzi, M.~Lenzinger, and B.~Schweizer.
\newblock Interface conditions for degenerate two-phase flow equations in one
  space dimension.
\newblock {\em Analysis}, 29:299--316, 2009.

\bibitem{CGP09}
C.~Canc{\`e}s, T.~Gallou\"et, and A.~Porretta.
\newblock Two-phase flows involving capillary barriers in heterogeneous porous
  media.
\newblock {\em Interfaces Free Bound.}, 11(2):239--258, 2009.

\bibitem{CGM15}
C.~Canc\`es, T.~O. Gallou\"et, and L.~Monsaingeon.
\newblock The gradient flow structure of immiscible incompressible two-phase
  flows in porous media.
\newblock {\em C. R. Acad. Sci. Paris S\'er. I Math.}, 353:985--989, 2015.

\bibitem{CGM17}
C.~Canc\`es, T.~O. Gallou\"et, and L.~Monsaingeon.
\newblock Incompressible immiscible multiphase flows in porous media: a
  variational approach.
\newblock {\em Anal. PDE}, 10(8):1845--1876, 2017.

\bibitem{CNV21}
C.~Canc\`es, F.~Nabet, and M.~Vohral{\'\i}k.
\newblock Convergence and a posteriori error analysis for energy-stable finite
  element approximations of degenerate parabolic equations.
\newblock {\em Math. Comp.}, 90(328):517--563, 2021.

\bibitem{CP12}
C.~Canc{\`e}s and M.~Pierre.
\newblock {An existence result for multidimensional immiscible two-phase flows
  with discontinuous capillary pressure field}.
\newblock {\em SIAM J. Math. Anal.}, 44(2):966--992, 2012.

\bibitem{Carman37}
P.~C. Carman.
\newblock Fluid flow through granular beds.
\newblock {\em Trans. Inst. Chem. Eng.}, 15:150--166, 1937.

\bibitem{CJ86}
G.~Chavent and J.~Jaffr\'e.
\newblock {\em Mathematical Models and Finite Elements for Reservoir
  Simulation}, volume~17.
\newblock North-Holland, Amsterdam, stud. math. appl. edition, 1986.

\bibitem{Chen01}
Z.~Chen.
\newblock Degenerate two-phase incompressible flow. {I}. {E}xistence,
  uniqueness and regularity of a weak solution.
\newblock {\em J. Differential Equations}, 171(2):203--232, 2001.

\bibitem{Costa06}
A.~Costa.
\newblock Permeability-porosity relationship: A reexamination of the
  {Kozeny-Carman} equation based on a fractal pore-space geometry assumption.
\newblock {\em Geophys. Res. Lett.}, 33(2), 2006.

\bibitem{Coussy04}
O.~Coussy.
\newblock {\em Poromechanics}.
\newblock John Wiley \& Sons, 2004.

\bibitem{kangourou}
J.~Droniou, R.~Eymard, T.~Gallou\"et, C.~Guichard, and R.~Herbin.
\newblock {\em The Gradient Discretisation Method}, volume~42 of {\em
  Math\'ematiques et Applications}.
\newblock Springer International Publishing, Cham, 2018.

\bibitem{GL12}
T.~Gallou{\"e}t and J.-C. Latch{\'e}.
\newblock Compactness of discrete approximate solutions to parabolic
  {PDE}s---application to a turbulence model.
\newblock {\em Commun. Pure Appl. Anal.}, 11(6):2371--2391, 2012.

\bibitem{Grisvard89}
P.~Grisvard.
\newblock Singularit\'es en elasticit\'e.
\newblock {\em Arch. Rational Mech. Anal.}, 107:157--180, 1989.

\bibitem{jha2014coupled}
B.~Jha and R.~Juanes.
\newblock Coupled multiphase flow and poromechanics: A computational model of
  pore pressure effects on fault slip and earthquake triggering.
\newblock {\em Water Resour. Res.}, 50(5):3776--3808, 2014.

\bibitem{kim2011stability}
J.~Kim, H.~A. Tchelepi, and R.~Juanes.
\newblock Stability and convergence of sequential methods for coupled flow and
  geomechanics: Fixed-stress and fixed-strain splits.
\newblock {\em Comput. Methods Appl. Mech. Engrg.}, 200(13-16):1591--1606,
  2011.

\bibitem{Kozeny27}
J.~Kozeny.
\newblock Uber kapillare {L}eitung der {W}asser in {B}oden.
\newblock {\em Royal Academy of Science, Vienna, Proc. Class I}, 136:271--306,
  1927.

\bibitem{Leverett41}
M.~C. Leverett.
\newblock Capillary behavior in porous solids.
\newblock {\em Transactions of the AIME}, 142, 1941.

\bibitem{Lions69}
J.-L. Lions.
\newblock {\em Quelques m\'ethodes de r\'esolution des probl\`emes aux limites
  non lin\'eaires}.
\newblock Dunod, 1969.

\bibitem{Mie11}
A.~Mielke.
\newblock A gradient structure for reaction-diffusion systems and for
  energy-drift-diffusion systems.
\newblock {\em Nonlinearity}, 24(4):1329--1346, 2011.

\bibitem{Pel-lecture}
M.~A. Peletier.
\newblock Variational modelling: Energies, gradient flows, and large
  deviations.
\newblock arXiv:1402.1990, 2014.

\bibitem{SRZRK19}
R.~Schulz, N.~Ray, S.~Zech, A.~Rupp, and P.~Knabner.
\newblock Beyond {K}ozeny-{C}arman: predicting the permeability in porous
  media.
\newblock {\em Transp. Porous Media}, 130:487--512, 2019.

\bibitem{seguin2019multi}
B.~Seguin and N.~J. Walkington.
\newblock Multi-component multiphase flow through a poroelastic medium.
\newblock {\em Journal of Elasticity}, 135:485--507, 2019.

\bibitem{showalter2000diffusion}
R.~E. Showalter.
\newblock Diffusion in poro-elastic media.
\newblock {\em J. Math. Anal. Appl.}, 251(1):310--340, 2000.

\bibitem{showalter2001partially}
R.~E. Showalter and N.~Su.
\newblock Partially saturated flow in a poroelastic medium.
\newblock {\em Discr. Cont. Dyn. Syst. B}, 1(4):403--420, 2001.

\bibitem{terzaghi1925erdbaumechanik}
K.~von Terzaghi.
\newblock {\em Erdbaumechanik auf bodenphysikalischer Grundlage}.
\newblock F. Deuticke, 1925.

\bibitem{ulm2004concrete}
F~J Ulm, Georgios Constantinides, and Franz~H Heukamp.
\newblock Is concrete a poromechanics materials?—a multiscale investigation
  of poroelastic properties.
\newblock {\em Materials and structures}, 37:43--58, 2004.

\bibitem{van2023mathematical}
C.~J. van Duijn and A.~Mikeli{\'c}.
\newblock Mathematical theory of nonlinear single-phase poroelasticity.
\newblock {\em J. Nonlinear Sci.}, 33(3):44, 2023.

\bibitem{van2019thermoporoelasticity}
C.~J. van Duijn, A.~Mikeli{\'c}, M.~F. Wheeler, and T.~Wick.
\newblock Thermoporoelasticity via homogenization: modeling and formal
  two-scale expansions.
\newblock {\em Internat. J. Engrg. Sci.}, 138:1--25, 2019.

\bibitem{vzenivsek1984existence}
A.~{\v{Z}}en{\'\i}{\v{s}}ek.
\newblock The existence and uniqueness theorem in {B}iot's consolidation
  theory.
\newblock {\em Aplikace matematiky}, 29(3):194--211, 1984.

\end{thebibliography}
\end{document}